\documentclass[11pt]{amsart}

\usepackage[a4paper,margin=1.15in]{geometry}
\usepackage{amsmath,amssymb,amsthm,amsfonts}
\usepackage{mathrsfs}
\usepackage{mathtools}
\usepackage{enumitem}
\usepackage{hyperref}
\usepackage[nameinlink,capitalise]{cleveref}
\usepackage{tikz-cd}
\hypersetup{
    colorlinks=true,
    linkcolor=blue,
    citecolor=blue,
    urlcolor=blue
}

\theoremstyle{plain}
\newtheorem{theorem}{Theorem}[section]
\newtheorem{proposition}[theorem]{Proposition}
\newtheorem{lemma}[theorem]{Lemma}

\newtheorem{conjecture}[theorem]{Conjecture}
\theoremstyle{definition}
\newtheorem{definition}[theorem]{Definition}

\theoremstyle{remark}
\newtheorem{remark}[theorem]{Remark}
\DeclareMathOperator{\End}{End}

\DeclareMathOperator{\Id}{\textnormal{Id}}
\DeclareMathOperator{\Herm}{\textnormal{Herm}}

\title[Geometry of non-Hermitian Yang--Mills moduli spaces]{Geometry of non-Hermitian Yang--Mills moduli spaces}

\author{Xingluan Wang}
\thanks{E-mail: wxlmath@mail.ustc.edu.cn. School of Mathematical Sciences,
University of Science and Technology of China, Hefei 230026, P. R. China.}

\begin{document}

\begin{abstract}
We study the moduli space of non-Hermitian Yang--Mills
connections over a compact Kähler manifold. Using normalized harmonic metrics, we construct a natural Hermitian metric on the unobstructed locus and
show that, near the Hermitian locus, the unobstructed
locus carries an almost hypercomplex structure which is compatible with the associated Riemannian metric.


\end{abstract}

\maketitle

\section{Introduction}

Non-Hermitian Yang--Mills (NHYM) connections were introduced by Kaledin and Verbitsky \cite{NHYM}. Let \(M\)
be a compact Kähler manifold and \(E\to M\) be a complex vector bundle.
A connection \(\nabla\) on \(E\) is called non-Hermitian Yang--Mills if its
curvature \(\Theta\) satisfies
\begin{equation}
\left\{
\begin{aligned}
\Lambda \Theta &= \lambda \,\operatorname{Id}_E,\\
\Theta &\in \Omega^{1,1}(\operatorname{End}E),
\end{aligned}
\right.
\end{equation}
for some constant \(\lambda\). This definition is standard; see \cite{Don,UY}. However, usually \(\nabla\) is
assumed to be compatible with some Hermitian metric on \(E\). This is why we
use the term ``non-Hermitian Yang--Mills'' to denote Yang--Mills connections
which are not necessarily Hermitian.

Kaledin and Verbitsky denote by
\(\mathcal M^s\) the moduli space of \((0,1)\)-stable NHYM
connections on \(E\), modulo the complex gauge group $G=\textnormal{Maps}(M, \textnormal{Aut} E)$. Here \((0,1)\)-stability
means that the holomorphic structure induced by the \((0,1)\)-part of the
connection is stable. This holomorphic bundle admits a Hermitian-Yang-Mills (HYM) connection. It is useful because the induced geometric structures  near the HYM locus are better behaved, which will be recalled below in our discussion.

The motivation for Kaledin and Verbitsky's study of NHYM connections comes from the following observation.

When \(c_1(E)=c_2(E)=0\),
HYM connections are flat \cite{Sim1}, and  the unitary flat connections are in one-to-one correspondence with those holomorphic
structures on \(E\) which make it a polystable holomorphic bundle
\cite{UY,Sim1}.

For an arbitrary bundle \(E\), the analogous correspondence is the Uhlenbeck--Yau theorem which states that polystable holomorphic structures on \(E\)
are exactly those which admit HYM connections. Thus it is natural to weaken
the flatness condition and consider instead all HYM
connections.

The moduli space of flat, but not necessarily unitary, bundles is a beautiful
subject, well studied in the literature \cite{Sim11,Sim22}. This
space has dimension twice that of the moduli space of unitary flat bundles and
carries a natural holomorphic symplectic form. Also, the generic part of
the moduli space of non-unitary flat bundles is equipped with a holomorphic
Lagrangian fibration over the space of unitary flat connections.

This analogy motivates the passage from HYM connections to NHYM connections. Namely, one keeps the HYM curvature equations but drops the Hermitian compatibility condition, just as one passes from unitary flat connections to flat connections which are not necessarily unitary. The resulting question is whether the corresponding NHYM moduli space admits analogues of the structures known in the flat case. Kaledin and Verbitsky showed that near the Hermitian locus, $\mathcal M^s$ has
dimension twice that of the moduli space of HYM connections
and is naturally equipped with a holomorphic symplectic form. As in the case of
flat bundles, the generic part of $\mathcal M^s$ admits a
holomorphic Lagrangian fibration over the space of HYM
connections \cite[Proposition 2.19, Proposition 2.25]{NHYM}.

When the base manifold is hyperkähler, it provides a further layer of structure. In this
case, the quaternionic structure induces an \(SU(2)\)-action on differential
forms. Kaledin and Verbitsky singled out autodual connections, namely
connections whose curvature is \(SU(2)\)-invariant, and showed that autodual connections
are NHYM \cite[Definition~3.6, Proposition~3.9]{NHYM}. Under the assumption that \(c_1(E)\) and \(c_2(E)\) are
\(SU(2)\)-invariant, every NHYM connection sufficiently close to the Hermitian
locus is autodual \cite[Theorem~3.11]{NHYM}. This led them to ask whether NHYM connections over hyperkähler manifolds
are necessarily autodual.

We now turn to the conjecture studied in this paper. Kaledin and
Verbitsky constructed a natural closed holomorphic two-form \(\Omega\) on the
smooth part of \(\mathcal M^s\). Since a hyperkähler
manifold, after choosing one of its complex structures, carries a canonical
holomorphic symplectic form \cite[Proposition 14.15]{Ein}, it is natural to ask whether the form \(\Omega\)
arises from a hyperkähler metric. 
\begin{conjecture}[\textnormal{\textbf{\cite[Conjecture 8.1]{NHYM}}}]
There exists a hyperkähler metric on \(\mathcal M^s\) such that \(\Omega\) is the
associated holomorphic symplectic form.
\end{conjecture}
When \(\dim_{\mathbb C}M=1\), the type condition is automatic and the NHYM
equation reduces to projective flatness. In the degree-zero case, this is the usual flatness equation, which is already well understood \cite{fujiki}. Therefore, in this paper we focus on the case \(\dim_{\mathbb C}M>1\). Note that the construction of the form \(\Omega\) is completely parallel to
the construction of a holomorphic symplectic form on the Hitchin--Simpson
moduli space \(\mathcal M^s_{\mathrm{DR}}\) of flat connections on \(E\)
\cite{Sim11,Sim22}. The analogue of the above conjecture for \(\mathcal M^s_{\mathrm{DR}}\) is
known.

Kaledin and Verbitsky state that there is also a formal reason for this expectation. After choosing a Hermitian
metric on the underlying bundle, the space \(\mathcal A\) of all connections
carries a natural infinite-dimensional hyperkähler structure, and the complex
gauge group acts compatibly with this structure. The corresponding complex
moment map is given by
$
\mu_\mathbb C(\nabla)=\Lambda\nabla^2,
$
so that it is exactly one of the NHYM equations. 

The hyperkähler reduction viewpoint also suggests an analogue of the Uhlenbeck–Yau theorem for NHYM bundles. This leads to the following conjecture.

\begin{conjecture}[\textnormal{\textbf{\cite[Conjecture 8.7]{NHYM}}}]
Let \(( E,\nabla)\) be a bundle with NHYM connection \(\nabla\).
Then there exists a harmonic metric \(h\) on \(E\) if and only if
\(E\) is a direct sum of \(\nabla\)-stable bundles.
Also, if \(E\) itself is \(\nabla\)-stable, then \(h\) is unique up to a constant factor.

\end{conjecture}

Pan, Shen, and Zhang later proved this conjecture in \cite{zhang}. Their result is the NHYM analogue of harmonic metrics on flat bundles \cite{Cor,Don1}. Their theorem provides the real moment map $\mu_\mathbb R(\nabla)=\Lambda \Xi$ required for hyperkähler reduction. However, this picture should be regarded as a heuristic rather than as a completed conclusion. In the NHYM setting, the relevant complex gauge action does not preserve any fixed background metric, and the formal ambient hyperkähler structure does not immediately descend to $\mathcal M^s$. We defer a detailed analysis of this approach to the final section.

Rather than giving a direct hyperkähler interpretation of the conjecture, we separate the two ingredients which would enter such a structure. Namely, we study the existence of a Hermitian metric on the unobstructed locus and, independently, the existence of an almost hypercomplex structure near the Hermitian locus. Our approach is modeled on Itoh’s construction \cite{Itoh2}. The relevant geometric objects are first constructed on local slices of the moduli space, then we check their compatibility with the transition maps on overlaps, so that the local constructions glue to globally defined
objects.

The main results are the following.

\begin{theorem}
Let \(M\) be a compact Kähler manifold and let \(E\to M\) be a complex vector bundle.
Let $\widehat{\mathcal M}^s$ denote the unobstructed locus of the moduli space of
\((0,1)\)-stable NHYM connections. Then $\widehat{\mathcal M}^s$ carries a
natural Hermitian metric.
\end{theorem}
\begin{theorem}
There exists a neighborhood of the Hermitian locus in $\widehat{\mathcal M}^s$ such that the neighborhood carries an almost hypercomplex structure
$(\mathcal I_\alpha,\mathcal J_\alpha,\mathcal K_\alpha)$ compatible with the associated Riemannian metric.
\end{theorem}

These results show that the conjectural hyperkähler structure, if it exists, cannot arise from the natural quaternionic almost complex structures constructed above. Consequently, the direct Kaledin--Verbitsky strategy is excluded.

The main difficulty in the NHYM case is that the relevant geometric structures cannot be obtained by directly descending fixed ambient structures. To obtain well-defined objects on the moduli space, one has to let the Hermitian metric and the almost complex structures vary with the base point and introduce correction terms ensuring compatibility with the transition maps between slices. These corrections are precisely what distinguishes the NHYM case from the classical  settings. They make the constructions well defined, but at the same time produce the closedness obstruction of the fundamental form and integrability obstructions of $\mathcal J_\alpha,\mathcal K_\alpha$.

This paper is organized as follows.

Section~2 reviews the description of the NHYM moduli space. In particular, we give the elliptic deformation complex, the slice expression, and the unobstructed locus \(\widehat{\mathcal M}^s\). 

Section~3 recalls the existence and uniqueness of harmonic metrics for stable
NHYM bundles. This result is one of the main
inputs in our construction of a Hermitian metric on
\(\widehat{\mathcal M}^s\). We also establish several elementary compatibility
properties of the Hermitian inner product under rescaling of Hermitian metrics and under gauge transformations.
These properties will be used later to prove that the locally defined metrics are compatible with the transition maps between slices.

Section~4 is devoted to the construction of the Hermitian metric on
\(\widehat{\mathcal M}^s\). We assign to each point of a local slice a normalized
harmonic metric and use it to define an \(L^2\)-inner product of tangent vectors. After proving the smooth dependence of this construction
and its gauge-equivariance, we show that the resulting local metrics are compatible
with the transition maps between slices. This yields a well-defined Hermitian metric
on the unobstructed locus.

Section~5 follows a similar strategy for almost complex structures. We start from the standard quaternionic structures on the ambient space of connections, as in the usual hyperkähler picture. However, this ambient triple does not in general preserve the tangent spaces of $\widehat{\mathcal M}^s$. It does so only along the Hermitian locus. We are therefore led to work in a neighborhood of the Hermitian locus. To extend the construction away from this locus, we introduce certain correction terms so that the resulting operators define an almost hypercomplex structure on the unobstructed locus under consideration.

In Section~6, we recall the hyperkähler reduction construction and discuss whether the standard hyperkähler reduction method can be applied to $\mathcal M^s$
. Indeed, $\mathcal M^s$ is only a subset of a formal hyperkähler space, and its tangent spaces are not, in general, preserved by the ambient complex structures. Therefore the hyperkähler structure of the ambient space cannot be inherited by $\mathcal M^s$
 through the usual reduction argument.

\section{Previous results on the moduli space of NHYM connections}
We begin by recalling some basic definitions and properties of the moduli space of NHYM connections \cite{NHYM}.

Suppose $E$ is a complex vector bundle over a compact Kähler manifold $M$. 
Let $\mathcal{A}$ denote the space of all connections on $E$. $\mathcal{A}_0$ is the subspace of NHYM connections in $\mathcal{A}$ whose curvature $\Theta$ is of Hodge type (1,1) and satisfies $\Lambda \Theta=\lambda \,\operatorname{Id}_E$. The Lie group $G=\textnormal{Maps}(M, \textnormal{Aut} E)$ acts on $\mathcal{A}$  by
$
u\cdot\nabla=u^{-1}\nabla u.
$
Equivalently, if \(\nabla'=\nabla+\eta\) with
\(\eta\in\Omega^1(\End E)\), then
$
u\cdot\nabla'
=
u^{-1}\nabla u+u^{-1}\eta u.
$
This action preserves the subset \(\mathcal A_0\subset\mathcal A\). The symbol \(u\) is used either for an element of the gauge group or for the induced gauge action, the intended meaning should be clear from the context. To simplify the exposition, we will always consider only NHYM connections with the constant $\lambda=0$.

Let $\mathcal{A}^s \subset \mathcal{A}_0$ be the open subset of (0,1)-stable NHYM connections.

\begin{definition}[\textnormal{\textbf{\cite{NHYM}}}]
Let $\mathcal{M}^s=\mathcal{A}^s/G$ be the moduli space of NHYM connections on $E$ endowed with the quotient topology.
\end{definition}
\begin{proposition}[\textnormal{\textbf{\cite[Proposition 2.7]{NHYM}}}]
The deformation complex $\mathcal C_\nabla$ associated with the NHYM connection $\nabla$, 
\begin{equation}0 \to \Omega^0(\End E) \xrightarrow{\,d_1\,} \Omega^1(\End E)\xrightarrow{\,d_2\,} \Omega^{2,0}(\End E) \oplus \Omega^{0,2}(\End E) \oplus \Omega^0(\End E)\end{equation}
is elliptic, where 
$d_{1}\xi=\nabla\xi,  \ d_{2}\beta
=
\bigl(
(\nabla\beta)^{2,0},
(\nabla\beta)^{0,2},
\Lambda\nabla\beta
\bigr)$. We write $\nabla=\nabla'+\nabla^{''}$ for the decomposition of $\nabla$ into its \((1,0)\)- and \((0,1)\)-components.
\end{proposition}

For later computations, we fix the following notational convention. Let
\(\nabla\) be a fixed NHYM connection. Unless otherwise stated, the operators
\(d_1\) and \(d_2\) always denote the operators in the deformation complex at
\(\nabla\). For \(\alpha\in\Omega^1(\End E)\), we write
\(d_{1,\alpha}\) and \(d_{2,\alpha}\) for the corresponding operators at
\(\nabla+\alpha\). Throughout the paper, the subscript \(\alpha\) will indicate that the
corresponding object is taken with respect to the base connection
\(\nabla+\alpha\).

\begin{lemma}For \(\xi\in\Omega^0(\End E), \beta\in\Omega^1(\End E)\),
\begin{equation}
d_{1,\alpha}\xi
=
\nabla\xi+[\alpha,\xi],
\end{equation}
\begin{equation}
d_{2,\alpha}\beta
=
\left(
\bigl(\nabla\beta+[\alpha,\beta]\bigr)^{2,0},
\bigl(\nabla\beta+[\alpha,\beta]\bigr)^{0,2},
\Lambda\bigl(\nabla\beta+[\alpha,\beta]\bigr)
\right).
\end{equation}
Here \([\cdot,\cdot]\) denotes the graded commutator of
\(\End E\)-valued forms.
\end{lemma}

Suppose $h$ is a Hermitian metric on $E$. The spaces \(\Omega^l(\End E)\) are endowed with the Sobolev
\(W^{k,2}\)-norms, denoted by \(|\cdot |_k\), where \(k\) is chosen sufficiently large relative to \(\dim M\).

The following proposition is obtained by the same Kuranishi-type argument as in \cite{koba1}.
\begin{proposition}[\textnormal{\textbf{\cite[(7.3.9)]{koba1}}}]
The slice of $\nabla$ in $\Omega^1(\End E)$ is \begin{equation}S^h_{\nabla,\varepsilon}=\{\alpha \in \Omega^1(\End E) : |\alpha|_k <\varepsilon, \nabla^{*,h} \alpha=0, (N_\alpha)^{2,0}=(N_\alpha)^{0,2}=\Lambda (N_\alpha)=0\},\end{equation} where $N_\alpha=\nabla \alpha+\alpha \wedge \alpha$.
\end{proposition}
Note that every $(0,1)$-stable $\nabla$ is simple, and obviously, if we choose another Hermitian metric $h'$, the following proposition still holds.
\begin{proposition}[\textnormal{\textbf{\cite[Theorem 7.3.17]{koba1}}}, \textnormal{\textbf{\cite[Proposition 2.3]{Itoh2}}}]
The natural map $p: S^h_{\nabla,\varepsilon} \to \mathcal{M}^s, \ \alpha \to [\nabla +\alpha]$ gives a homeomorphism of a neighborhood of $0$ in $S^h_{\nabla,\varepsilon}$ onto a neighborhood of $[\nabla]$ in $\mathcal{M}^s$.
\end{proposition}

After replacing
\(\varepsilon\) by a smaller positive number, we may assume that the above
homeomorphism is defined on the whole slice \(S^h_{\nabla,\varepsilon}\).

\begin{proposition}[\textnormal{\textbf{\cite[Corollary 2.9]{NHYM}}}]
$\mathcal{M}^s$ is a complex-analytic space. 
\end{proposition}
Let \(\End_0E\subset \End E\) denote the trace-free endomorphism bundle.
$H^i(\widetilde {\mathcal C}_\nabla)$ is the $i$-th cohomology group of the complex  for the trace-free deformation complex
\begin{equation}
0
\longrightarrow
A^0(\End_0E)
\xrightarrow{d_{1}}
A^1(\End_0E)
\xrightarrow{d_{2}}
A^{2,0}(\End_0E)\oplus A^{0,2}(\End_0E)\oplus A^0(\End_0E)
 .
\end{equation}

We denote the unobstructed locus by
$
\widehat{\mathcal M}^s
=
\left\{
[\nabla]\in\mathcal M^s
\;\middle|\;
H^2(\widetilde {\mathcal C}_\nabla)=0
\right\}
.$
\(\mathcal M^s\) is nonsingular at every point of \(\widehat{\mathcal M}^s\), and
$
T_{[\nabla]}\mathcal M^s \cong H^1(\mathcal C_\nabla).
$
This is the standard smoothness criterion for such elliptic deformation
problems; compare with Kobayashi's corresponding statements 
\cite[Corollary 7.3.32,\ Corollary 7.4.18]{koba1}. The condition \(H^2(\widetilde C_\nabla)=0\) is used here only as a sufficient
smoothness criterion. It is not meant to characterize all smooth points of
\(\mathcal M^s\); smooth points may exist outside \(\widehat{\mathcal M}^s\).
This distinction does not affect the results of the present paper.

In the following sections, we shall mainly work with \(\widehat{\mathcal M}^s\subset \mathcal M^s\).

\section{Harmonic metrics on NHYM bundles}

In this section, we recall the existence of harmonic metrics for semisimple NHYM bundles and prove several properties needed later.

Let $E$ be a complex vector bundle over a compact Kähler manifold $(M,\omega)$ and let $\nabla$ be a connection on $E$. Given a Hermitian metric $h$ on $E$, there is a unique decomposition
$\nabla = \nabla_h + \psi^{\nabla}_h$,
where $\nabla_h$ is an $h$-unitary connection and $\psi_h^{\nabla} \in \Omega^1(\End(E))$ is self-adjoint with respect to $h$.

For later reference, we first introduce the notation and conventions used throughout the sequel. Readers already familiar with them may skip it.

\begin{definition}
Let \(h\) be a Hermitian metric on \(E\). We use the following conventions. Throughout, Hermitian pairings are taken to be linear in the first argument.
\begin{enumerate}[label=\textup{(\roman*)}]
\item For $s,t \in \Gamma(E)$, the fiberwise Hermitian pairing on
\(E\) is denoted by \(h(s,t)\).

\item For $A \in \Gamma(\End E)$, its fiberwise \(h\)-adjoint is denoted
by \(A^{*,h}\) and is defined by $h(As,t)=h(s,A^{*,h}t).$ The same notation is used for \(\End E\)-valued forms:
$
(\eta\otimes A)^{*,h}
=
\bar\eta\otimes A^{*,h}.
$

\item The fiberwise Hermitian pairing on \(\End E\) is denoted by $
\langle A,B\rangle_{\End,h}=\operatorname{tr}(B^{*,h}A)$, $A,B \in \Gamma(\End E)$.

\item For \(F\)-valued forms
\(\beta_1,\beta_2\) of the same degree where $F=E$ or $F=\End E$, we write 
$\langle \beta_1,\beta_2\rangle_{F,h}$
for the pointwise pairings induced by the fixed Kähler metric on \((M,\omega)\) and by
the corresponding fiber metrics. The associated \(L^2\)-pairing is $
(\beta_1,\beta_2)_{F,h}
=
\int_M \langle \beta_1,\beta_2\rangle_{F,h}\,\frac{\omega^n}{n!}.
$
\item For a connection \(\nabla\) on \(E\), define a connection
\(\nabla^{\dagger_h}\) by $d\bigl(h(s,t)\bigr)
=
h(\nabla s,t)+h(s,\nabla^{\dagger_h}t), \ s,t\in\Gamma(E).
$

\item For \(F=E\) or \(F=\End E\), we denote by \(\nabla^{*,h}\) the
\(L^2\)-formal adjoint of the induced covariant derivative
\(\nabla:\Omega^k(F)\to\Omega^{k+1}(F)\), characterized by $
(\nabla\beta_1,\beta_2)_{F,h}
=
(\beta_1,\nabla^{*,h}\beta_2)_{F,h},\ \beta_1\in\Omega^k(F),\ 
\beta_2\in\Omega^{k+1}(F)$.
\item The formal adjoint of
$
d_{1,\nabla}:\Omega^0(\End E)\longrightarrow \Omega^1(\End E)
$
with respect to \((\ ,\ )_{\End,h}\) is denoted by
$
d_{1,\nabla}^{*,h}:\Omega^1(\End E)\longrightarrow \Omega^0(\End E),
$
and is characterized by 
$
\bigl(d_{1,\nabla}\beta_1,\beta_2\bigr)_{\End,h}$ $
=
\bigl(\beta_1,d_{1,\nabla}^{*,h}\beta_2\bigr)_{\End,h}$, 
$
\beta_1\in\Omega^0(\End E),\ 
\beta_2\in\Omega^1(\End E)$.
\item For \(u\in G\), the pullback Hermitian metric \(u^*h\) is defined by $
(u^*h)(s,t)=h(u   s,u t)$, $ s,t\in\Gamma(E)$.
\end{enumerate}
\end{definition}

When applied to endomorphisms or endomorphism-valued forms, the symbol
\(^{*,h}\) denotes the fiberwise Hermitian adjoint. When applied to a differential
operator, it denotes the formal \(L^2\)-adjoint with respect to the indicated metric. When no confusion is possible, the subscripts \(E,\ \End\) and \(h\) are suppressed.

With these conventions, we have $\nabla_h=\frac{1}{2}(\nabla+\nabla^{\dagger_h})$, $\psi^{\nabla}_h=\frac{1}{2}(\nabla-\nabla^{\dagger_h})$.

\begin{definition}[\textnormal{\textbf{\cite[(2.13)]{zhang}}}]
Let \((E,\nabla)\) be an NHYM bundle over a compact Kähler manifold
\((M,\omega)\), and let \(h\) be a Hermitian metric on \(E\). We say that \(h\) is a harmonic metric if it
satisfies either, and hence both, of the following equivalent definitions:
\begin{enumerate}[label=\textup{(\roman*)}]
\item  \(h\) satisfies
    $
    \nabla_h^{*,h}\psi^{\nabla}_h=0.
    $

    \item Let \(\Xi_h\) be the
    pseudocurvature of \((E,\nabla,h)\), then
    $
    \Lambda \Xi_h=0.
    $
\end{enumerate}
\end{definition}

\begin{theorem}[\textnormal{\textbf{\cite[Theorem 1.2]{zhang}}}]\label{ha}
Let $(X,\omega)$ be a compact Kähler manifold, and let $(E,\nabla)$ be a NHYM bundle over $X$. Then $(E,\nabla)$ admits a harmonic metric if and only if it is semisimple. 
\end{theorem}
\begin{remark}[\textnormal{\textbf{\cite[Remark 2.4]{zhang}}}, \textnormal{\textbf{\cite[Remark 8.4]{NHYM}}}]
The definition of \(\nabla\)-simplicity used here agrees with
Kaledin--Verbitsky's notion of \(\nabla\)-stability. Consequently,
\(\nabla\)-semisimplicity means that the NHYM bundle is a direct sum of
\(\nabla\)-stable NHYM bundles, as in Conjecture~8.7 of \cite{NHYM}.
Moreover, every \((0,1)\)-stable NHYM connection is \(\nabla\)-stable.
\end{remark}

We restrict to (0,1)-stable NHYM connections $\nabla$. In this case, the corresponding harmonic metric is unique up to a positive scalar multiple.

For each $\nabla \in \widehat{\mathcal M}^s$, let \([h_\nabla]\) be the positive-scale class of the harmonic metric associated with \(\nabla\). Since \(c>0\) is a constant, multiplying \(h\) by \(c\) does not change the fiberwise adjoint of an endomorphism. Indeed, if \(A\in \End(E_x)\), then for any \(s,t\in E_x\),
\begin{equation}
(ch)(As,t)
=
c\,h(As,t)
=
c\,h(s,A^{*,{h}}t)
=
(ch)(s,A^{*,{h}}t).
\end{equation}
Hence \(A^{*,{ch}} = A^{*,{h}}\). The pointwise Hermitian inner product on \(\End E\) is
$
\langle A,B\rangle_{ch}
=
\operatorname{tr}(B^{*,{ch}}A)
=
\operatorname{tr}(B^{*,h}A)
=
\langle A,B\rangle_h.
$ Therefore the pointwise Hermitian pairing on \(\End E\), and hence the induced pairing on \(\Omega^1(\End E)\), is unchanged.
\begin{lemma}\label{lem1}
Suppose \(h\) is a Hermitian metric on \(E\), \(c>0\) is a constant, and \(u \in G\), then for every $\xi \in \Omega^0(\End E)$, \(\eta,\eta_1,\eta_2 \in \Omega^1(\End E)\) and $A \in \Gamma(\End E)$,

\begin{enumerate}[label=\textup{(\roman*)}]
\item 
$
(\eta_1,\eta_2)_{c h} = (\eta_1,\eta_2)_h,\ 
d_{1,\nabla}^{*,ch} = d_{1,\nabla}^{*,h}.
$
\item $(u^{-1} A u)^{*,u^*h}=u^{-1} A^{*,h} u$.
\item $
(u\cdot \eta_1,\, u\cdot \eta_2)_{u^*h} = (\eta_1,\eta_2)_h, \ 
d_{1,u\cdot \nabla}^{*,u^*h} (u \cdot \eta)
=
u\cdot ( d_{1,\nabla}^{*,h}\eta).
$

\end{enumerate}
As a consequence, the gauge action preserves the space of harmonic forms: $u\cdot \mathcal H_\nabla^1=\mathcal H_{u\cdot \nabla}^1 .$

\end{lemma}
\begin{proof}
(i) is clear since $
(\xi,d_{1,\nabla}^{*,ch}\eta)_{ch}=(d_{1,\nabla}\xi,\eta)_{ch}
=
(d_{1,\nabla}\xi,\eta)_h
=
(\xi,d_{1,\nabla}^{*,h}\eta)_h,
$
and (ii) and (iii) follow by 
\begin{equation}
\begin{aligned}
\bigl((u^{-1}Au)s,t\bigr)_{u^*h}
&=
\bigl(u(u^{-1}Au)s,ut\bigr)_h 
=
(Aus,ut)_h 
=
(us,A^{*,h}ut)_h \\
&=
\bigl(us,u(u^{-1}A^{*,h}u)t\bigr)_h 
=
\bigl(s,(u^{-1}A^{*,h}u)t\bigr)_{u^*h}.
\end{aligned}
\end{equation}
\begin{equation}
\begin{aligned}
\langle u\cdot A,u\cdot B\rangle_{u^*h}
&=
\operatorname{tr}\left((u^{-1}Bu)^{*,{u^*h}}(u^{-1}Au)\right)
=
\operatorname{tr}\left(u^{-1}B^{*,h}u\,u^{-1}Au\right)\\
&=
\operatorname{tr}\left(u^{-1}B^{*,h}Au\right)
=
\operatorname{tr}(B^{*,h}A)
=
\langle A,B\rangle_h.
\end{aligned}
\end{equation}
\begin{equation}
\begin{aligned}
(u\cdot\xi,\ d_{1,u\cdot \nabla}^{*,u^*h}(u\cdot\eta))_{u^*h}
&=
(d_{1,u\cdot \nabla}(u\cdot\xi),\ u\cdot\eta)_{u^*h}
=
(u\cdot d_{1,\nabla}\xi,\ u\cdot\eta)_{u^*h}\\
&=
(d_{1,\nabla}\xi,\eta)_h
=
(\xi,d_{1,\nabla}^{*,h}\eta)_h
=
(u\cdot\xi,\ u\cdot d_{1,\nabla}^{*,h}\eta)_{u^*h}.
\end{aligned}
\end{equation}
\end{proof}

\section{Hermitian metric on the moduli space of NHYM connections}

In this section, we construct a Hermitian metric on $\widehat{\mathcal M}^s$.
The construction is carried out locally on gauge slices and consists of the following steps:
\begin{enumerate}[label=\textup{(\roman*)}]
\item Define a Hermitian pairing on each local slice by using the harmonic metric associated with the corresponding $(0,1)$-stable NHYM connection.

\item Show that the local Hermitian pairing is well-defined and varies smoothly on each slice.

\item Prove that these local Hermitian pairings are compatible with the transition maps between slices.

\end{enumerate}

Consequently, the local pairings descend to a Hermitian metric
on $\widehat{\mathcal M}^s$.

The construction is analogous to Itoh's construction in \cite[Section~3]{Itoh2}, except that in the present setting the Hermitian metric is chosen pointwise on the moduli space rather than fixed once and for all, since the complex gauge action does not preserve a fixed Hermitian metric, a
single background metric cannot be used in the construction. Instead, we choose
the Hermitian metric pointwise, in such a way that it transforms equivariantly
under gauge transformations up to a positive scalar. Since such a scalar
rescaling does not change the induced pairing on \(\End E\)-valued forms, the
resulting slice metric is invariant under gauge-induced transition maps, see Lemma~\ref{lem4}\textup{(i)}.

We begin by recalling the slice of the NHYM moduli space. The following definitions are independent of the choice of representative element in the harmonic metric class $[h_\nabla]$. 
\begin{equation}
S^{h_\nabla}_{\nabla,\varepsilon}=\{\alpha \in \Omega^1 : |\alpha|_k <\varepsilon, \nabla^{*,h_\nabla} \alpha=0,  (N_\alpha)^{2,0}=(N_\alpha)^{0,2}=\Lambda (N_\alpha)=0\},\end{equation}
\begin{equation}T_\alpha S^{h_\nabla}_{\nabla,\varepsilon}=\{\beta \in \Omega^1 : \nabla^{*,h_\nabla} \beta=0, d_{2,\alpha}(\beta)=0\}.\end{equation}

In particular,
\begin{equation}T_\nabla S^{h_\nabla}_{\nabla,\varepsilon}=\{\beta \in \Omega^1 : \nabla^{*,h_\nabla} \beta=0, d_{2,\nabla}(\beta)=0\}=\ker d^*_{1,\nabla}\cap \ker d_{2,\nabla}\cong \mathcal{H}^1_{\nabla}.\end{equation}

\begin{lemma}\label{lem4}
Fix a reference Hermitian metric \(h\) on \(E\). For every $(0,1)$-stable NHYM connection \(\nabla\), let \([h_\nabla]\) denote the class of harmonic metrics associated with \(\nabla\). Then there exists a unique representative $ h_\nabla^{\mathrm{n}}\in [h_\nabla]$
satisfying the normalization condition
$
\int_M
\log\det (h^{-1}$ $ \cdot h_\nabla^{\mathrm{n}})\frac{\omega^n}{n!}
=0.
$ 
Thus the assignment
\[
\nabla\longmapsto h_\nabla^{\mathrm{n}}
\]
is an assignment to actual Hermitian metrics on \(E\). Moreover, it satisfies the following properties:
\begin{enumerate}[label=\textup{(\roman*)}]
\item  
$
h_{u\cdot \nabla}^{\mathrm{n}}=
c(u,\nabla) u^*h_\nabla^{\mathrm{n}}
$
for a uniquely determined positive constant \(c(u,\nabla)>0\). 

\item  After possibly shrinking \(\varepsilon_0\), the assignment $
\nabla\longmapsto h_\nabla^{\mathrm{n}} $
is smooth on \(S_{\nabla_0,\varepsilon_0}\).
\end{enumerate}

\end{lemma}
\begin{proof}

We first prove (i). In this step only, \(k\) denotes a Hermitian metric. Let $k=h_\nabla^{\mathrm{n}},\ k'=u^* k$, then by definition, \begin{equation}d (k'(s,t))=d k(u s,u t)=k(\nabla_k(u  s),u t)+k(u  s,\nabla_k(u  t)).\end{equation}

Since $u ((u \cdot \nabla_k)s)=u  ((u^{-1}\circ \nabla_k \circ u)s)=\nabla_k (u s)$,
we have \[d (k'(s,t))=k(u ((u \cdot \nabla_k)s),u t)+k(u s,u ((u \cdot \nabla_k)t))=k'((u \cdot \nabla_k)s,t)+k'(s,(u \cdot \nabla_k)t).\]

Similarly, \(u\cdot\psi_k\) is \(k'\)-self-adjoint, since
\begin{equation}
k'((u\cdot\psi_k)s,t)
=
k(\psi_k(us),u t)
=
k(us,\psi_k(ut))
=
k'(s,(u\cdot\psi_k)t).
\end{equation}
If we write $u \cdot \nabla=(u \cdot \nabla)_{k'}+(u \cdot \psi)_{k'}$, then $(u \cdot \nabla)_{k'}=u \cdot \nabla_{k}$ and $(u \cdot \psi)=u \cdot \psi_{k}$.

Consequently, 
\begin{equation}
(u\cdot \nabla_k)^{*,k'}(u\cdot \psi_k)
=
u\cdot (\nabla_k^{*,k}\psi_k).
\end{equation}

Thus, if \(k\) is harmonic for \(\nabla\), then \(u^*k\) is harmonic for \(u\cdot \nabla\). The constant $c(u,\nabla)$ is uniquely determined by the normalization condition.

It remains to prove (ii). Consider the slice $S^{h_{\nabla_0}^{\mathrm{n}}}_{\nabla_0,\varepsilon_0}$ and let 
$
h_0=h_{\nabla_0}^{\mathrm{n}}
$. 
The elements of the slice are of the form  
$
\nabla_\alpha=\nabla_0+\alpha\in S^{h_{0}}_{\nabla_0,\varepsilon_0}.
$

Every Hermitian metric sufficiently close to \(h_0\) can be written uniquely as
$
h_f=h_0e^f,
$
where
$
f=f^{*,{h_0}}\in \Herm(E,h_0)=\{f\in \End(E)|f^{*,{h_0}}=f\}.
$ 
Since harmonic metrics are unique only up to positive constant multiples, we impose the normalization
$
\int_M \operatorname{tr}(f)\,\frac{\omega^n}{n!}=0.
$
After Sobolev completion, set
\begin{equation}
B_k=
\left\{
f\in W^{k,2}(\Herm(E,h_0))
\;\middle|\;
\int_M \operatorname{tr}(f)\,\frac{\omega^n}{n!}=0
\right\}.
\end{equation}
Let $
F(\nabla_\alpha,f)=e^{f/2}(\nabla_{\alpha,h_f})^{*,h_f}\psi^{\nabla_{\alpha}}_{h_f}e^{-f/2} $
denote the harmonic metric equation for \(\nabla_\alpha\) with respect to the metric \(h_f=h_0e^f\) with $F(\nabla_0,0)=0$. If an endomorphism \(A\) is self-adjoint with respect to \(h_f\), then
\(e^{f/2} A e^{-f/2}\) is self-adjoint with respect to \(h_0\). For \(k\) sufficiently large,
\begin{equation}
F:
S^{h_0}_{\nabla_0,\varepsilon_0}\times B_k
\longrightarrow B_{k-2}
\end{equation}
is a smooth map.

By \cite[Lemma 2.1]{zhang2}, we obtain
$LF
:=
D_f F \big|_{(\nabla_0,0)}=
-\frac12
(\nabla_0^{\dagger_{h_0}})^{*,h_0}\nabla_0^{\dagger_{h_0}}: B_k
\longrightarrow B_{k-2}$.

It is elliptic, self-adjoint, and Fredholm of index zero. Its kernel consists of \(\nabla_0^{\dagger_h}\)-parallel
\(h_0\)-self-adjoint endomorphisms satisfying the above trace normalization. 

 Since \(\nabla_0\) is simple,
$
\End_{\nabla_0}(E)=\mathbb C\Id.
$
The equations $s=s^{*,h_0}$ and $\nabla_0^{\dagger_{h_0}}s=0$ imply $\nabla_0 s=0$, and thus $s\in \mathbb C \Id$.
Thus the \(h_0\)-self-adjoint part of the kernel is \(\mathbb R\Id\), and the normalization removes this
one-dimensional kernel. Hence
$
\ker LF=0.
$
Since \(LF\) is Fredholm of index zero, it follows that
$
LF:B_k\longrightarrow B_{k-2}
$
is an isomorphism.

The implicit function theorem therefore gives, after shrinking the slice if necessary, a unique map from $S^{h_0}_{\nabla_0,\varepsilon_0}$ to $B_k$, 
$
\alpha\longmapsto f(\alpha)
$
such that
$
F(\nabla_\alpha,f(\alpha))=0$, and
$f(0)=0.
$
By elliptic regularity,
\(f(\alpha)\) is in fact smooth.

Therefore the normalized harmonic metric
$
h_{\alpha}=h_0e^{f(\alpha)}
$
varies smoothly on the local slice. It is easy to prove $h_{\alpha}$ satisfies the normalization condition since $
\log\det\bigl(h^{-1}h_\alpha\bigr)
=
\log\det\bigl(h^{-1}h_0\bigr)+\operatorname{tr}\bigl(f(\alpha)\bigr)
$.

\end{proof}

Throughout the rest of the paper, the metric entering the definition of the slice is always the normalized representative \(h_\nabla^{\mathrm n}\). Thus, whenever \(h_\nabla\) appears in the slice notation, it is understood to mean \(h_\nabla^{\mathrm n}\). When no confusion is likely, we suppress this metric from the notation and write
$
S_{\nabla,\varepsilon}$ for $S_{\nabla,\varepsilon}^{h_{\nabla}^{\mathrm{n}}}.
$

\begin{lemma}
Fix $S^{h_\nabla}_{\nabla,\varepsilon}$, let $S^{h_{\nabla'}}_{\nabla',\varepsilon'}$ be another slice centered at $\nabla'=\nabla +\alpha_0 \in S^{h_\nabla}_{\nabla,\varepsilon}$, then there is a smooth mapp
\begin{equation}
u: U_{\nabla \nabla'}:=S^{h_\nabla}_{\nabla,\varepsilon} \cap p^{-1}(p(S^{h_{\nabla'}}_{\nabla',\varepsilon'})) \to G, \ \alpha \mapsto u_\alpha
\end{equation}
such that $u_\alpha \cdot(\nabla+\alpha) \in  S^{h_{\nabla'}}_{\nabla',\varepsilon'}$ for $\alpha \in U_{\nabla \nabla'}$ near $\alpha_0$, and $u_{\alpha_0}\cdot (\nabla+\alpha_0)=\nabla+\alpha_0$.
\end{lemma}

\begin{proof}
The proof is similar to the proof of \cite[Theorem 7.3.17]{koba1}. Define \begin{equation}\begin{aligned}\Psi:W^{k,2}(\Omega^1(\End E)) \times U_{k+1} &\longrightarrow U_{k-1}\\ (\alpha,\xi)&\mapsto  d_{1,\nabla'}^{*,h_{\nabla'}}(\exp(\xi)\cdot (\nabla+\alpha)-\nabla'),\end{aligned}\end{equation} where we set $
U_k
=
\left\{
\xi\in W^{k,2}(\End(E))
\;\middle|\;
\int_M \operatorname{tr}(\xi)\,\frac{\omega^n}{n!}=0
\right\},
$
then $\left. D_\xi \Psi \right|_{(\alpha_0,0)}=d^{*,h_{\nabla'}}_{\nabla'}d_{\nabla'}: U_{k+1} \to U_{k-1}$ is elliptic and self-adjoint. It is an isomorphism.
Therefore the Banach implicit function theorem gives, for $\alpha$ near \(\alpha_0\), a unique smooth map
$
\xi=\xi(\alpha)\in U_{k+1}$, with $
\xi(\alpha_0)=0,
$
such that
$
\Psi(\alpha,\xi(\alpha))=0.
$
Set
$
u_\alpha=\exp(\xi(\alpha)),
$
then
$
d_{1,\nabla'}^{*,h_{\nabla'}}
\left(
u_\alpha\cdot(\nabla+\alpha)-\nabla'
\right)=0,
$
so
$
u_\alpha\cdot(\nabla+\alpha)\in S_{\nabla',\varepsilon'}^{h_{\nabla'}}. 
$ The smoothness assertion follows from elliptic regularity.
\end{proof}

The preceding lemma gives an explicit formula for the transition map between the slices:
$
\tau_{\nabla\nabla'}
:
U_{\nabla\nabla'}
\longrightarrow
U_{\nabla'\nabla}
$
by
$
\tau_{\nabla\nabla'}(\nabla+\alpha)
=
u_\alpha\cdot(\nabla+\alpha).
$

\

The tangent space at \(\nabla+\alpha\) is $T_{\nabla+\alpha}S_{\nabla,\varepsilon}=\ker d^*_{1,\nabla} \cap \ker d_{2,\alpha}$ , we define an inner product by
\begin{equation}
\langle \beta_1,\beta_2\rangle^\mathrm{h}_{\nabla+\alpha}
=(\beta_1^\mathrm{h},\beta_2^\mathrm{h})_{\End E,h_{\alpha}}
=
\int_M \langle \beta_1^\mathrm{h},\beta_2^\mathrm{h}\rangle_{\End E,h_{\alpha}}\,\frac{\omega^n}{n!}.
\end{equation}
Here \(\beta_1^\mathrm{h}, \beta_2^\mathrm{h} \in \ker d_{2,\alpha} \cap \ker d_{1,\alpha}^* \) are the harmonic parts with respect to the
splitting
$
\Omega^1(\End (E))
=
\operatorname{Im} d_{1,\alpha}
\oplus
\ker d_{1,\alpha}^* 
$ \cite[(3.4)]{Itoh2}. The definition is consistent with \cite[(3.3)]{Itoh2}. The Hermitian pairing
$
\langle \cdot,\cdot \rangle^{\mathrm h}_{\nabla+\alpha}
$
is also denoted by
$
(\cdot,\cdot)_{h_\alpha^{\mathrm h}}.
$
The corresponding Riemannian metric $g_\alpha^{\mathrm h}$ is defined by taking the real part.

It remains only to prove that this construction is independent of the choice of slice. The proof of the following proposition follows from \cite[(3.5)]{Itoh2}.
\begin{proposition}
Let
$
\beta \in T_{\nabla+\alpha} S_{\nabla,\varepsilon},
$
and let \(\alpha(t)\) be a smooth curve in \(S_{\nabla,\varepsilon}\) with
\(\alpha(0)=\alpha\) and \(\dot\alpha(0)=\beta\).
Then the differential of $\tau$ at $\nabla+\alpha$ in the direction $\beta$ is $(\tau_{\nabla \nabla'})_{*,\nabla+\alpha}
:
T_{\nabla+\alpha}S_{\nabla,\varepsilon}
\longrightarrow
T_{\tau(\nabla+\alpha)}S_{\nabla',\varepsilon'}$ given by
\begin{equation}
\tau_{*,\nabla+\alpha}(\beta)=d_{1,{\nabla'+\alpha'}}\psi+u_\alpha(\beta)
\end{equation}
for some $\psi \in \Omega^0(\End E)$.
\end{proposition}

Thus
\begin{equation}
\begin{aligned}
\bigl\langle \tau_*\beta_1,\tau_*\beta_2\bigr\rangle^\mathrm{h}_{\nabla'+\alpha'}
&=
\int_M
\bigl(
(\tau_*\beta_1)^\mathrm{h},
(\tau_*\beta_2)^\mathrm{h}
\bigr)_{\nabla'+\alpha'}
\,\frac{\omega^n}{n!} =
\int_M
\bigl(
u_\alpha(\beta_1^\mathrm{h}),
u_\alpha(\beta_2^\mathrm{h})
\bigr)_{u_\alpha\cdot (\nabla+\alpha)}
\,\frac{\omega^n}{n!} \\[2mm]
&=
\int_M
\bigl(
\beta_1^\mathrm{h},
\beta_2^\mathrm{h}
\bigr)_{\nabla+\alpha}
\,\frac{\omega^n}{n!} =
\bigl\langle \beta_1,\beta_2\bigr\rangle^\mathrm{h}_{\nabla+\alpha}.
\end{aligned}
\end{equation}

Thus it is invariant under changes of slice. Here \(\langle\cdot,\cdot\rangle_{\nabla+\alpha}\) indicates that the pairing is
defined using $h^{\mathrm{n}}_{\alpha}$. Consequently, the local
pairings glue to a well-defined Hermitian metric on
\(\widehat{\mathcal M}^s\) with respect to the complex structure induced by $I$.

\begin{remark}
The metric constructed above should be regarded as a Hermitian metric which is
not expected, in general, to be Kähler.
There is no a priori reason for its associated fundamental form to be closed. 
Indeed, this point can be seen by comparing the construction with the proof of
\cite[Proposition 4.2]{Itoh2}. In that proof, the \(L^2\)-inner product is defined using a
single fixed background metric throughout the local model. This fixedness is a
key reason why the corresponding fundamental form is closed. By contrast, in
the present construction the harmonic metric entering the definition varies with
the base point of the slice. Consequently, when one differentiates the associated
fundamental form, the first variation of this harmonic metric produces additional
terms. There is no apparent reason for these extra terms to cancel in general.
\end{remark}

\section{Almost hypercomplex structure on the moduli space of NHYM connections}

In this section, we construct a quaternionic-type triple of operators $\mathcal I_{\alpha},\mathcal J_{\alpha},\mathcal K_{\alpha}$ on the NHYM moduli space. More precisely, $\mathcal  I_{\alpha}$ defines a genuine complex structure, whereas $\mathcal J_{\alpha}$ and $\mathcal K_{\alpha}$ are only almost complex structures and are not asserted to be integrable. Our emphasis will not be on the complex structure $\mathcal  I_{\alpha}$ alone. We will treat the triple $\mathcal I_{\alpha},\mathcal J_{\alpha},\mathcal K_{\alpha}$ simultaneously.

We study almost complex structures on the moduli space in the same way as the above section: namely, we first define $\mathcal I_{\alpha},\mathcal J_{\alpha},\mathcal K_{\alpha}$ on each local slice and then check that they are compatible.

Our strategy will be to construct the almost complex structures $\mathcal I_{\alpha},\mathcal J_{\alpha},\mathcal K_{\alpha}:T_\alpha S_{\nabla,\varepsilon}
\longrightarrow
T_\alpha S_{\nabla,\varepsilon}$, as the
composition 
\begin{equation}\mathcal L_{\alpha}:T_\alpha S_{\nabla,\varepsilon}
\xrightarrow{\;Q_\alpha\;}
\mathcal H_\alpha^1
\xrightarrow{\;\tilde{L}_{\alpha}^{H}\;}
\mathcal H_\alpha^1
\xrightarrow{\;Q_\alpha^{-1}\;}
T_\alpha S_{\nabla,\varepsilon},\ \mathcal L_{\alpha} \in \{\mathcal I_{\alpha},\mathcal J_{\alpha},\mathcal K_{\alpha}\}
\end{equation}

 Following Hitchin’s hyperkähler construction, we introduce three formal candidate complex structure operators.
 For every Hermitian metric $h_\alpha$, define real-linear maps \(I_{\alpha},J_{\alpha},K_{\alpha}:\Omega^1(\End E) \to \Omega^1(\End E)\)  by
\begin{equation}
I_{\alpha}(a,b)=(ia,ib),\
J_{\alpha}(a,b)=\bigl(i b^{*,h_\alpha},-i a^{*,h_\alpha}\bigr),\
K_{\alpha}(a,b)=\bigl(-b^{*,h_\alpha},a^{*,h_\alpha}\bigr),
\end{equation}
where we use the decomposition $\Omega^1({\End E})=\Omega^{0,1}(\End E)\oplus \Omega^{1,0}(\End E)$.

It follows by a straightforward argument that \(I_\alpha^2=J_\alpha^2=K_\alpha^2=-\Id\), $I_\alpha J_\alpha=K_\alpha$ and they satisfy the quaternionic identities. These operators are equivariant with respect to the
gauge transformation \(u_\alpha\). We shall also need the following proposition.

\begin{proposition}Let $g_\alpha=\operatorname{Re}h_\alpha$ be the corresponding Riemannian metric, then $
g_\alpha(L_\alpha s,$ \\$L_\alpha t)$ $=g_\alpha(s,t),\ 
L_\alpha\in\{I_\alpha,J_\alpha,K_\alpha\}.$
\end{proposition}

\begin{proof}
Let $s=(a,b),t=(c,d)\in \Omega^1(\End E).$ It is immediate that \(I_\alpha\) preserves \(g_\alpha\). Since
\(K_\alpha=I_\alpha J_\alpha\), it remains to prove the assertion for
\(J_\alpha\).

We use the identity $
\bigl(\eta^{*,{h_\alpha}},\theta^{*,{h_\alpha}}\bigr)_{h_\alpha}
=
\overline{(\eta,\theta)_{h_\alpha}}
$, then
\begin{equation}
\begin{aligned}
g_\alpha(J_\alpha s,J_\alpha t)
&=
\operatorname{Re}\left[
\bigl(
(i b^{*,h_\alpha},-i a^{*,h_\alpha}),
(i d^{*,h_\alpha},-i c^{*,h_\alpha})
\bigr)_{h_\alpha}\right]\\
&=
\operatorname{Re}
\left[
(b^{*,h_\alpha},d^{*,h_\alpha})_{h_\alpha}
+
(a^{*,h_\alpha},c^{*,h_\alpha})_{h_\alpha}
\right]\\
&=
\operatorname{Re}
\left[
\overline{(b,d)_{h_\alpha}}
+
\overline{(a,c)_{h_\alpha}}
\right]\\
&=
g_\alpha(s,t).
\end{aligned}
\end{equation}
\end{proof}

We begin with the almost complex structure on $\mathcal H^1_\alpha$. Take \(I^H_{\alpha}=I_{\alpha}\big|_{\mathcal H^1_\alpha},J^H_{\alpha}=J_{\alpha}\big|_{\mathcal H^1_\alpha}$, $K^H_{\alpha}=K_{\alpha}\big|_{\mathcal H^1_\alpha}\). In general, they do not satisfy the quaternionic relations. Nevertheless, since $I$ preserves harmonic forms, they still satisfy some identities. For instance,
$
I_\alpha J_\alpha^H=K_\alpha^H, J^H_\alpha I^H_\alpha=-I_\alpha J^H_\alpha
$
and similarly for the other cyclic relations. In particular, when \(\alpha=0\), the three operators are defined on $T_\nabla S_{\nabla,\varepsilon}^{h_\nabla}\cong \mathcal H^1_\nabla$.

\begin{proposition}
 \(I^H_{\nabla}\) preserves the tangent space, i.e. 
$
I^H_{\nabla}\bigl(T_\nabla S_{\nabla,\varepsilon}\bigr)
\subset
T_\nabla S_{\nabla,\varepsilon}.
$
If, in addition, \(\nabla\) is unitary with respect to \(h_\nabla\), then \(J^H_\nabla\) and \(K^H_\nabla\) also preserve the tangent space, i.e.
$
J^H_\nabla\bigl(T_\nabla S_{\nabla,\varepsilon}\bigr)
\subset
T_\nabla S_{\nabla,\varepsilon},
K^H_\nabla\bigl(T_\nabla S_{\nabla,\varepsilon}\bigr)
\subset
T_\nabla S_{\nabla,\varepsilon}.
\ 
$
\end{proposition}

\begin{proof}
For notational convenience, we suppress the subscript \(h_\nabla\) from the
fiberwise adjoint.
Let \(\beta=(a,b)\in T_\nabla S_{\nabla,\varepsilon}=
\ker d_{1,\nabla}^{*}\cap \ker d_{2,\nabla}\), then
\begin{equation}
\nabla' b=0,\
\nabla''a=0,\
\Lambda(\nabla'a+\nabla''b)=0,
\
d_{1,\nabla}^{*}(a,b)=0.
\end{equation}

Since both
\(d_{1,\nabla}^{*,h_\nabla}\) and \(d_{2,\nabla}\) are complex-linear, we have
$
I^H_\nabla\beta\in T_\nabla S_{\nabla,\varepsilon}.
$

Since \(\nabla\) is unitary with respect to \(h_\nabla\), we have
$
(\nabla'a)^*=\nabla''(a^*),
(\nabla''b)^*=\nabla'(b^*),
(\nabla''a)^*$ $=\nabla'(a^*),
(\nabla'b)^*=\nabla''(b^*),
$
and the Kähler identities give
$
(\nabla'')^*a=-i\Lambda\nabla'a,
(\nabla')^*b=i\Lambda\nabla''b.
$
Therefore
\begin{equation}
0=d_{1,\nabla}^{*,h_\nabla}(a,b)=(\nabla'')^{*,h_\nabla} a +(\nabla')^{*,h_\nabla}b=i\Lambda\nabla''b-i\Lambda\nabla'a.
\end{equation}

Set
$
J^H_\nabla\beta=(a_J,b_J)=(i b^*,-i a^*).
$
Since
\(\nabla''a=0\) and \(\nabla'b=0\), we compute,
\begin{equation}
\nabla' b_J
=
-i\nabla'(a^*)
=
-i(\nabla''a)^*
=
0,\
\nabla'' a_J
=
i\nabla''(b^*)
=
i(\nabla'b)^*
=
0.\end{equation}
\begin{equation}
\begin{aligned}
\Lambda(\nabla'a_J+\nabla''b_J)
&=
\Lambda\bigl(i\nabla'(b^*)-i\nabla''(a^*)\bigr)\\
&=
i\Lambda\bigl((\nabla''b)^*-(\nabla'a)^*\bigr)\\
&=
i\bigl(\Lambda(\nabla''b-\nabla'a)\bigr)^*=0.
\end{aligned}
\end{equation}
\begin{equation}
\begin{aligned}
d_{1,\nabla}^{*}(J^H_\nabla\beta)
&=
(\nabla'')^*(i b^*)+(\nabla')^*(-i a^*)\\
&=
\Lambda\nabla'(b^*)+\Lambda\nabla''(a^*)\\
&=
\Lambda\bigl((\nabla''b)^*+(\nabla'a)^*\bigr)\\
&=
\bigl(\Lambda(\nabla''b+\nabla'a)\bigr)^*=0.
\end{aligned}
\end{equation} 

Hence
$
J^H_\nabla\beta\in T_\nabla S_{\nabla,\varepsilon}.
$

Finally, since \(K^H_\nabla=I^H_\nabla J^H_\nabla\), and both
\(I^H_\nabla\) and \(J^H_\nabla\) preserve
\(T_\nabla S_{\nabla,\varepsilon}\) under the Hermitian assumption, the
same is true for \(K^H_\nabla\). This proves the proposition.
\end{proof}

In view of the preceding proposition, we shall henceforth restrict our attention to slices centered at the Hermitian locus, namely slices $S_{\nabla,\varepsilon} $ with \(\nabla\) Hermitian.

Thus, we are led to study almost complex structures on the open subset \begin{equation}\mathcal U
=
\bigcup
p\bigl(S_{\nabla,\varepsilon_\nabla}\bigr)\subset \mathcal M^s,\end{equation}
where the union is taken over all $(0,1)$-stable HYM connections satisfying $H^2(\widetilde {\mathcal C}_\nabla)=0$.

The transition maps on \(\mathcal U\) considered here satisfy the following properties.
\begin{lemma}
Let \(\nabla\) and \(\nabla'\) be HYM connections, and consider the
Hermitian-centered slices
$
S_{\nabla,\varepsilon}^{h_\nabla}$ and $
S_{\nabla',\varepsilon'}^{h_{\nabla'}}.
$
On the non-empty overlap
$
U_{\nabla\nabla'}
=
S_{\nabla,\varepsilon}^{h_\nabla}
\cap
p^{-1}\bigl(p(S_{\nabla',\varepsilon'}^{h_{\nabla'}})\bigr),
$
there exists $V \subset U_{\nabla\nabla'}$ and a smooth map
$
u:V \to G,\ 
\alpha\mapsto u_\alpha,
$
such that
$
u_\alpha\cdot(\nabla+\alpha)\in S_{\nabla',\varepsilon'}^{h_{\nabla'}}
$. 
Thus the transition map is
$
\tau_{\nabla\nabla'}(\nabla+\alpha)
=
u_\alpha\cdot(\nabla+\alpha).
$
\end{lemma}

Thus far, we have constructed three almost complex structures on the tangent space at the base point $\nabla$. These structures, however, do not admit a natural extension to the tangent spaces at points of the form $\nabla+\alpha$.
The idea is to use the harmonic space as an intermediate model and define the corrected candidate almost complex structure by conjugation through $Q_\alpha=\mathbb H_\alpha|_{T_\alpha S_{\nabla,\varepsilon}}$, where $
\mathbb H_\alpha:\Omega^1(\End E)\longrightarrow \mathcal H_{\alpha}^1
$
denotes the harmonic projection, namely:
\begin{equation}\label{cs}
\mathcal L_{\alpha}:T_\alpha S_{\nabla,\varepsilon}
\xrightarrow{\;Q_\alpha\;}
\mathcal H_\alpha^1
\xrightarrow{\;\tilde{L}_{\alpha}^{H}\;}
\mathcal H_\alpha^1
\xrightarrow{\;Q_\alpha^{-1}\;}
T_\alpha S_{\nabla,\varepsilon},\ 
\mathcal L_\alpha\in\{\mathcal I_\alpha,\mathcal J_\alpha,\mathcal K_\alpha\}.
\end{equation}

The use of the harmonic projection and of \(Q_\alpha^{-1}\) generally destroys integrability. Thus the construction does not imply the integrability of \(\mathcal J_\alpha\) and \(\mathcal K_\alpha\).
We now explain how each of the maps in this construction is obtained. The next two propositions describe the operators \(Q_\alpha\) and \(\tilde{L}_{\alpha}^{H}\), respectively.

\begin{proposition}
$Q_\alpha=\mathbb H_\alpha|_{T_\alpha S_{\nabla,\varepsilon}}:T_\alpha S_{\nabla,\varepsilon} \to \mathcal H_{\alpha}^1$ is invertible for small $\alpha$.
\end{proposition}

\begin{proof}
$Q_\alpha$ is well-defined since every 
$
\beta\in T_\alpha S_{\nabla,\varepsilon}
\subset \ker d_{2,\alpha}
$
admits a decomposition
$\beta=\beta^{\mathrm{h}}+d_{1,\alpha}\xi,\ 
\beta^{\mathrm{h}}\in \mathcal H_{\alpha}^{1},
\xi\in \Omega^{0}(\End E).$

Given $\gamma \in \mathcal H^1_\alpha$, we need to find $\beta \in T_\alpha S_{\nabla,\varepsilon}=\ker d^*_{1,\nabla} \cap \ker d_{2,\alpha}$ such that $Q_\alpha(\beta)=\gamma$. Clearly, $\beta$  has  the form $\beta=\gamma+d_{1,\alpha}\xi$ and $d_{2,\alpha}\beta=d_{2,\alpha}\gamma+d_{2,\alpha}d_{1,\alpha}\xi=0$.

It remains only to impose the fixed slice condition
$
\nabla^{*,h_{\nabla}}\beta=0.
$
Substituting
$
\beta
$
into this condition, we obtain
\begin{equation}
\nabla^{*,h_{\nabla}}\bigl(\gamma+d_{1,\alpha}\xi\bigr)=0.
\end{equation}
Define $B_\alpha=d_{1,\nabla}^{*,h_{\nabla}}d_{1,\alpha}:\Omega^0(\End E)\to \Omega^0(\End E) $.
It remains to solve the equation
\[
B_\alpha\xi=-\nabla^{*,h_{\nabla}}\gamma.
\]
\textbf{Claim: }$\widetilde{B}_\alpha=B_\alpha|_{\Omega^{0}_{\perp}}:\Omega^{0}_{\perp}(\End E)\to \Omega^{0}_{\perp}(\End E)$ is an isomorphism for small $\alpha$ where $\Omega^{0}_{\perp}(\End E)=(\ker d_{1,\nabla})^{\perp}=(\mathbb C \Id)^{\perp}.$

Since $\widetilde{B}_0=d_{1,\nabla}^{*,h_{\nabla}}d_{1,\nabla}:\Omega^{0}_{\perp} \to \Omega^{0}_{\perp}$ is an isomorphism, for small $\alpha$, $\widetilde{B}_\alpha$ is also an isomorphism. This is the standard fact that the set of invertible bounded operators between Banach spaces is open in the operator norm topology. A point worth noting is that \(\Omega^0_{\perp}\) does not depend on the choice of the Hermitian metric since $\nabla$ here is simple.

Note that $\nabla^{*,h_{\nabla}}\gamma \in \Omega^{0}_{\perp}(\End E)$  since $(\nabla^{*,h_{\nabla}}\gamma,\eta)=(\gamma,\nabla \eta)=0 \textnormal{ for } \eta \in \ker d_{1,\nabla}.$
Thus $\xi=-\widetilde{B}^{-1}_\alpha \nabla^{*,h_{\nabla}}\gamma$ is well-defined.

To see the injectivity of $Q_\alpha$, suppose $Q_\alpha(\beta)=0$. There exists $\gamma \in \mathcal H^1_\alpha, \xi_{\perp} \in \Omega^0_{\perp}$ and $\xi_0 \in \ker \nabla =\mathbb C \Id$ such that $\beta=\gamma+d_{1,\alpha}\xi_{\perp}+d_{1,\alpha}\xi_0$. Then $0=Q_\alpha(\beta)=\gamma$ and clearly $d_{1,\alpha}\xi_0=0$, $0= d^{*,h_{\nabla}}_{1,\nabla}\beta=d^{*,h_{\nabla}}_{1,\nabla}d_{1,\alpha}\xi_{\perp}=\widetilde{B}_\alpha \xi_{\perp} \Rightarrow \xi_{\perp}=0\Rightarrow \beta=0$.

\end{proof}

\begin{proposition}

Set $
S_{\alpha}=\mathbb H_\alpha J^H_\alpha:\mathcal H_\alpha^1 \to \mathcal H_\alpha^1$, and $
R_{\alpha}=\bigl(-S_{\alpha}^2\bigr)^{1/2},$ 
then \(\tilde{J}^H_{\alpha}=S_{\alpha} R_{\alpha}^{-1}:\mathcal H_\alpha^1\to \mathcal H_\alpha^1\) is a well-defined almost complex structure on the underlying real vector space of
\(\mathcal H^1_\alpha\).
\end{proposition}

\begin{proof}
The proposition follows from the following steps.
\begin{enumerate}[label=\textup{(\roman*)}]
\item
$S_\alpha=\mathbb H_\alpha J^H_\alpha:\mathcal H^1_\alpha \to\mathcal H^1_\alpha$ is well-defined.

\item $S_\alpha$ is skew-adjoint, i.e. $S_{\alpha}^{*,g_\alpha^{\mathrm{h}}}=-S_\alpha$. 

Indeed, for $\beta_1,\beta_2 \in \mathcal H^1_\alpha$, $\langle S_\alpha \beta_1,\beta_2 \rangle_{g_\alpha}^\mathrm{h} =\langle\mathbb H_\alpha J^H_\alpha \beta_1,\beta_2 \rangle_{g_\alpha}^\mathrm{h} =\langle J^H_\alpha \beta_1,\beta_2 \rangle_{g_\alpha}^\mathrm{h} =\langle \beta_1,-J^H_\alpha \beta_2 \rangle_{g_\alpha}^\mathrm{h} =\langle \beta_1,- \mathbb H_\alpha J^H_\alpha \beta_2 \rangle_{g_\alpha}^\mathrm{h} =-\langle \beta_1,S_\alpha \beta_2 \rangle_{g_\alpha}^\mathrm{h} .$

\item $-S^2_\alpha$ is positive definite self-adjoint. 

Obviously, $S^{*,g^\mathrm{h}_\alpha}_\alpha S=-S^2_\alpha$. $\langle -S^2_\alpha \beta,\beta\rangle_{g_\alpha}^\mathrm{h}=||S_\alpha \beta||^2 \geq 0$. If $-S^2_\alpha \beta=0$, then $S_\alpha \beta=0$. The following claim implies $\beta=0$.

\textbf{Claim: }$S_\alpha$ is still an isomorphism for $\alpha$ sufficiently small.

Let $\Phi_\alpha
=
\mathbb H_\alpha\big|_{\mathcal H_\nabla^1}:\mathcal H_\nabla^1
\longrightarrow
\mathcal H_\alpha^1$, we define \(\widetilde{S}_\alpha:\mathcal H_\nabla^1 \to
\mathcal H_\nabla^1\) by the composition
$\widetilde{S}_\alpha
=
\Phi_\alpha^{-1}\circ S_\alpha\circ \Phi_\alpha,
$
that is,
$
\mathcal H_\nabla^1
\xrightarrow{\;\Phi_\alpha\;}
\mathcal H_\alpha^1
\xrightarrow{\;S_\alpha\;}
\mathcal H_\alpha^1
\xrightarrow{\;\Phi_\alpha^{-1}\;}
\mathcal H_\nabla^1$.

Here we show
$
\Phi_\alpha
$
is an isomorphism for $\alpha$ sufficiently small. If $\Phi_\alpha(v)=0$, then $v=\mathbb H_\nabla\big|_{\mathcal H_\nabla^1}(v)-\Phi_\alpha(v)=(\mathbb H_\nabla\big|_{\mathcal H_\nabla^1}-\Phi_\alpha)v$. Comparing the norms of both sides gives $v=0$. Moreover, Hodge theory for elliptic complexes implies both \(\mathcal H^1_\nabla\) and \(\mathcal H^1_\alpha\) are finite-dimensional vector spaces of the same dimension. Hence \(\Phi_\alpha\) is an isomorphism.  Since $\widetilde S_0=S_0=J^H_\nabla:\mathcal H^1_\nabla \to \mathcal H^1_\nabla$ is an isomorphism, for small $\alpha$, $\widetilde{S}_\alpha$ is still an isomorphism.  Therefore $
S_\alpha=\Phi_\alpha\circ \widetilde S_\alpha\circ \Phi_\alpha^{-1}$ is also an isomorphism.
\end{enumerate}
Since $\mathcal H^1_\alpha $ is finite-dimensional and $-S^2_\alpha$ is positive definite self-adjoint, let $R_\alpha$ denote the unique positive square root of $-S^2_\alpha$ and $R_\alpha$ commutes with $S_\alpha$ since \(R_\alpha\) commutes with any bounded linear operator that commutes with \(-S^2_\alpha\) by \cite[Theorem 6.6.4]{linear}.
\[(\widetilde{J}^H_\alpha)^2=S_{\alpha} R_{\alpha}^{-1}S_{\alpha} R_{\alpha}^{-1}=S_{\alpha}S_{\alpha} R_{\alpha}^{-1} R_{\alpha}^{-1}=S^2_\alpha (-S^2_\alpha)^{-1}=-\Id.\] 

Moreover, we have $
(\widetilde{J}^{H}_{\alpha})^{*,g_\alpha^{\mathrm{h}}}
=
(S_{\alpha}R_{\alpha}^{-1})^{*,g_\alpha^{\mathrm{h}}}
=
R_{\alpha}^{-1}S_{\alpha}^{*,g_\alpha^{\mathrm{h}}}
=
-R_{\alpha}^{-1}S_{\alpha}
=
-\widetilde{J}^{H}_{\alpha}.
$ Thus the operator $\widetilde{J}^{H}_{\alpha}$ is $g_\alpha^{\mathrm{h}}$-orthogonal.
\end{proof}

The above proposition gives $\mathcal I_\alpha^2=\mathcal J_\alpha^2=\mathcal K_\alpha^2=-\Id$.
\begin{proposition}
The three operators $\mathcal I_\alpha,\mathcal J_\alpha,\mathcal K_\alpha$ defined in \eqref{cs} are almost complex structures satisfying $\mathcal I_\alpha\mathcal J_\alpha=\mathcal K_\alpha$, and thus constitute an almost hypercomplex structure on $T_\alpha S_{\nabla,\varepsilon}$.
\end{proposition}

\begin{proof}
We first verify that $\mathcal I_\alpha,\mathcal J_\alpha,\mathcal K_\alpha$ satisfy the quaternionic relations and then check that these structures are compatible with the transition maps between the slices.
To distinguish these operators, we attach the subscripts \(I,J,K\) to \(S_\alpha\) and \(R_\alpha\), according to the corresponding almost complex structures.

First, since \(I\) is a complex structure already, it preserves the harmonic forms, and it follows immediately that \(\mathbb H_\alpha I_\alpha=I_\alpha \mathbb H_\alpha=I_\alpha^H \mathbb H_\alpha,S_{I,\alpha}=\mathbb H_\alpha I^H_\alpha=I^H_\alpha, S^2_{I,\alpha}=-\Id_{\mathcal H^1_\alpha},R_{I,\alpha}=\Id_{\mathcal H^1_\alpha},\widetilde{I}^H_{\alpha}=I^H_\alpha\).

\begin{enumerate}[label=\textup{(\roman*)}]
\item $S_{K,\alpha}=I^H_\alpha S_{J,\alpha}$.

 $
S_{K,\alpha}
=
\mathbb H_{\alpha}K_{\alpha}^{H}
=
\mathbb H_{\alpha}I_{\alpha}J_{\alpha}^{H}
=
I_{\alpha}\mathbb H_{\alpha}J_{\alpha}^{H}
=
I_{\alpha}^{H}S_{J,\alpha} .$
\item $S_{J,\alpha}I_\alpha^H
=
-I_\alpha^H S_{J,\alpha}$.

$
S_{J,\alpha}I_\alpha^H =
\mathbb H_\alpha J_\alpha^H I_\alpha^H =
-\mathbb H_\alpha I_\alpha J_\alpha^H =
-I_\alpha^H\mathbb H_\alpha J_\alpha^H =
-I_\alpha^H S_{J,\alpha} .
$
\item $-S_{K,\alpha}^2
=
-S_{J,\alpha}^2$.

$
S_{K,\alpha}^2=
\bigl(I_\alpha^H S_{J,\alpha}\bigr)
\bigl(I_\alpha^H S_{J,\alpha}\bigr)
=
I_\alpha^H S_{J,\alpha}I_\alpha^H S_{J,\alpha}
=
I_\alpha^H\bigl(-I_\alpha^H S_{J,\alpha}\bigr)S_{J,\alpha}
=
-\bigl(I_\alpha^H\bigr)^2S_{J,\alpha}^2
=
S_{J,\alpha}^2.
$
\item
$R_{K,\alpha}
=
R_{J,\alpha}$ by uniqueness of the positive square root.
\end{enumerate}

Then we have $
\widetilde K_\alpha^H
=
S_{K,\alpha}R_{K,\alpha}^{-1}
=
I_\alpha^H S_{J,\alpha}R_{J,\alpha}^{-1}
=
I_\alpha^H \widetilde J_\alpha^H=\widetilde{I}_\alpha^H \widetilde J_\alpha^H,
$ and $
\widetilde{J}^{H}_{\alpha} \widetilde{I}^{H}_{\alpha}
= S_{J,\alpha} R^{-1}_{J,\alpha} \widetilde{I}^{H}_{\alpha} 
= S_{J,\alpha} \widetilde{I}^{H}_{\alpha} R^{-1}_{J,\alpha} 
= - \widetilde{I}^{H}_{\alpha} S_{J,\alpha} R^{-1}_{J,\alpha} 
= - \widetilde{I}^{H}_{\alpha} \widetilde{J}^{H}_{\alpha}.
$

 Thus $\mathcal I_\alpha\mathcal J_\alpha=Q_\alpha^{-1} \widetilde{I}^H_\alpha \widetilde{J}^H_\alpha  Q_\alpha=Q_\alpha^{-1} \widetilde{K}^H_\alpha Q_\alpha=\mathcal K_\alpha$ and $\mathcal I_\alpha\mathcal J_\alpha=-\mathcal J_\alpha\mathcal I_\alpha$.

It is enough to prove the compatibility for \(\mathcal J_\alpha\), since the arguments for the other almost complex structures are analogous. Thus it suffices to establish the commutativity of the following diagram:
\begin{equation}
\begin{tikzcd}[column sep=large,row sep=large]
T_{\alpha}S_{\nabla,\varepsilon}
\arrow[r, "\mathcal J^{\nabla}_{\alpha}"]
\arrow[d, "(\tau_{\nabla\nabla'})_{*,\alpha}"']
&
T_{\alpha}S_{\nabla,\varepsilon}
\arrow[d, "(\tau_{\nabla\nabla'})_{*,\alpha}"]
\\
T_{\alpha'}S_{\nabla',\varepsilon'}
\arrow[r, "\mathcal J^{\nabla'}_{\alpha'}"]
&
T_{\alpha'}S_{\nabla',\varepsilon'}
\end{tikzcd}
\end{equation}
where $\mathcal J_{\alpha}=Q_{\alpha}^{-1}\widetilde{L}_{\alpha}^{H} Q_{\alpha}$ and the subscript indicates the slice on which the operator is defined. Since \(\mathcal H_\alpha^1\) is a complex subspace, we have
$\bigl(\mathcal H_\alpha^1\bigr)^{\perp,g_\alpha}=\bigl(\mathcal H_\alpha^1\bigr)^{\perp,h_\alpha}.
$

The following identities establish the commutativity of the diagram. Here and below,
\(u_\alpha\) acts on tangent vectors by the gauge action. Thus, for any operator
\(L\), we use the shorthand
$
(u_\alpha\circ L)(\beta)=u_\alpha\cdot(L\beta)$, and $
(L\circ u_\alpha)(\beta)=L(u_\alpha\cdot\beta).
$

\begin{enumerate}[label=\textup{(\roman*)}]
\item $\mathbb H_{\alpha'}\circ u_\alpha = u_\alpha  \circ \mathbb H_{\alpha}$. It implies $Q_{\alpha'}\circ \tau_*=u_\alpha \circ Q_\alpha$.

We need to prove $u_\alpha (\mathcal H^1_\alpha) \subset \mathcal H^1_{\alpha'}$ and $u_\alpha ((\mathcal H^1_\alpha)^{\perp,g_\alpha}) \subset (\mathcal H^1_{\alpha'})^{\perp,g_{\alpha'}}$. It is clear since $g_{\alpha'}(u_\alpha \beta_1,u_\alpha \beta_2)=g_\alpha(\beta_1,\beta_2)$ and $u_\alpha (\mathcal H^1_{\alpha})=\mathcal H^1_{\alpha'}$.

\item $J^H_{\alpha'}\circ u_\alpha=u_\alpha \circ J^H_{\alpha}$.

$
J_{\alpha'}^{H}u_{\alpha}(a,b)=
J_{\alpha'}^{H}
\bigl(
u_{\alpha}^{-1}a u_{\alpha},
 u_{\alpha}^{-1}b u_{\alpha}
\bigr)=
\left(
i\bigl( u_{\alpha}^{-1}b u_{\alpha}\bigr)^{*,h_{\alpha'}},
-i\bigl( u_{\alpha}^{-1}a u_{\alpha}\bigr)^{*,h_{\alpha'}}
\right)$ \\
= $
(
i u_{\alpha}^{-1}b^{*,h_{\alpha}} u_{\alpha},$$
-i u_{\alpha}^{-1}a^{*,h_{\alpha}} u_{\alpha}
)$ $=
 u_{\alpha}
\bigl(
ib^{*,h_{\alpha}},
-ia^{*,h_{\alpha}}
\bigr)$ $=
 u_{\alpha}J_{\alpha}^{H}(a,b).
$
\item $S_{\alpha'} \circ u_\alpha=u_\alpha \circ S_{\alpha}$.

$
S_{\alpha'}u_\alpha=
\mathbb H_{\alpha'}\, J^H_{\alpha'}\, u_\alpha  =
\mathbb H_{\alpha'}\,u_\alpha\,J^H_{\alpha}  =
u_\alpha\,\mathbb H_{\alpha}\,J^H_{\alpha}  =
u_\alpha\,S_{\alpha}.
$
$S_{\alpha'}=u_\alpha\,S_{\alpha}u_\alpha^{-1}$.
$-S^2_{\alpha'}=u_\alpha\,(-S^2_{\alpha})u_\alpha^{-1}$.
Since $u_\alpha$ is an isometry, and the positive square root is unique, then $R_{\alpha'}=u_\alpha\,R_{\alpha}u_\alpha^{-1}$.

\item $\widetilde{J}^{H}_{\alpha'}\circ u_\alpha=u_\alpha \circ \widetilde{J}^{H}_\alpha$.

$
\widetilde{J}^{H}_{\alpha'}\,u_\alpha
=
S_{\alpha'}\,R_{\alpha'}^{-1}\,u_\alpha  =
u_\alpha\,S_{\alpha}\,u_\alpha^{-1}
\circ
u_\alpha\,R_{\alpha}^{-1}\,u_\alpha^{-1}
\circ
u_\alpha  =
u_\alpha\,S_{\alpha}\,R_{\alpha}^{-1}  =
u_\alpha\,\widetilde{J}^{H}_\alpha.
$
\end{enumerate}
Using the above identities, we compute,
\begin{equation}
\begin{aligned}
Q_{\alpha'}
\Bigl(
(\tau_{\nabla\nabla'})_{*,\alpha}\,
\mathcal J_\alpha^\nabla \beta
\Bigr)
&=
Q_{\alpha'}
\Bigl(
u_\alpha \cdot (\mathcal J_\alpha^\nabla\beta)
+
d_{1,\alpha'}\psi_{J_\alpha^\nabla\beta}
\Bigr) \\
&=
(\mathbb H_{\alpha'}\circ u_\alpha) (\mathcal J_\alpha^\nabla\beta) \\
&=
(u_\alpha \circ \mathbb H_\alpha)(\mathcal J_\alpha^\nabla\beta) \\
&=
(u_\alpha \circ Q_\alpha)
(
Q_\alpha^{-1}\widetilde{J}_\alpha^H Q_\alpha\beta
) \\
&=
(u_\alpha \circ \widetilde{J}_\alpha^H Q_\alpha)\beta \\
&=
(\widetilde{J}_{\alpha'}^H \circ u_\alpha\circ  Q_\alpha)\beta \\
&=
\widetilde{J}_{\alpha'}^H Q_{\alpha'}
\Bigl(
(\tau_{\nabla\nabla'})_{*,\alpha}\beta
\Bigr)\\
&=
 Q_{\alpha'}  Q_{\alpha'}^{-1}\widetilde{J}_{\alpha'}^H
 Q_{\alpha'}\Bigl(
(\tau_{\nabla\nabla'})_{*,\alpha}\beta
\Bigr)\\
&= Q_{\alpha'} \mathcal J_{\alpha'}^{\nabla'}\Bigl(
(\tau_{\nabla\nabla'})_{*,\alpha}\beta
\Bigr).
\end{aligned}
\end{equation}

Since $Q_{\alpha'}$ is an isomorphism for small $\alpha'$, the diagram is commutative. We now verify that the almost complex structure constructed above is smooth. It is enough to show that the assignment $\alpha \longmapsto \mathcal J_\alpha$ is smooth. By construction, this reduces to the smoothness of the following two ingredients. 

First, the map $\alpha \longmapsto \mathbb H_\alpha
$
is smooth. This follows from \cite[Theorem 5.1]{Te}, applied to the family of formally self-adjoint strongly elliptic operators \(L_s=\Delta_\alpha^1\). In that theorem, \(F_s\) denotes the harmonic projector, which corresponds in our notation to \(\mathbb H_\alpha\).

Second, the map
$
\alpha \longmapsto R_\alpha
$
is smooth. Indeed, \(R_\alpha\) is defined as the principal square root of the positive self-adjoint operator
$
-S_\alpha^2.
$
After choosing a local orthonormal frame, \(-S_\alpha^2\) is represented by a smooth family of real symmetric positive definite matrices. Therefore the smoothness of \(R_\alpha\) follows from \cite[Theorem 1.1]{Tay}, applied to \(Q=-S_\alpha^2\).

\end{proof}

\begin{remark}
The underlying Riemannian metric of the Hermitian metric is compatible with the almost hypercomplex structure since
\begin{equation}
\begin{aligned}
\langle \mathcal L_\alpha\beta_1, \mathcal L_\alpha\beta_2 \rangle_{g_\alpha}^{\mathrm{h}}
&=
\langle
Q_\alpha Q_\alpha^{-1} \widetilde{L}^H_\alpha Q_\alpha \beta_1,
Q_\alpha Q_\alpha^{-1} \widetilde{L}^H_\alpha Q_\alpha \beta_2
\rangle^{\mathrm{h}}_{g_\alpha}\\
&=
\langle
 \widetilde{L}^H_\alpha Q_\alpha \beta_1,
\widetilde{L}^H_\alpha Q_\alpha \beta_2
\rangle^{\mathrm{h}}_{g_\alpha}\\
&=
\langle
Q_\alpha \beta_1,
Q_\alpha \beta_2
\rangle^{\mathrm{h}}_{g_\alpha}\\
&=
\langle  \beta_1, \mathcal \beta_2 \rangle_{g_\alpha}^{\mathrm{h}}.
\end{aligned}
\end{equation}
\end{remark}

\section{Toward the Hyperkähler Reduction Method}

In this section, we discuss whether the hyperkähler quotient method proposed in \cite[Section 8]{NHYM} can be applied to the NHYM moduli space. 

We begin by recalling the hyperkähler quotient construction of
\cite[Theorem 3.2]{hklr}, which provides a mechanism for producing new hyperkähler
manifolds from hyperkähler manifolds equipped with tri-Hamiltonian group
actions.
\begin{theorem}
Let \((M,g,I,J,K)\) be a hyperkähler manifold and let \(G\) be a compact Lie group acting on \(M\) by triholomorphic isometries. Suppose that the action admits an equivariant hyperkähler moment map $
    \mu=(\mu_1,\mu_2,\mu_3):M\longrightarrow \mathfrak g^*\otimes\mathbb R^3 .$
If \(0\) is a regular value of \(\mu\) and \(G\) acts freely on \(\mu^{-1}(0)\), then $\mu^{-1}(0)/G$
is a smooth hyperkähler manifold. Its metric and three Kähler forms are induced by restricting the corresponding tensors on \(M\) to the horizontal distribution orthogonal to the \(G\)-orbits.
\end{theorem}

This construction has several standard applications. Kronheimer's construction realizes
hyperkähler ALE spaces \cite{Kron1}. In the same spirit, Nakajima realizes quiver
varieties as hyperkähler quotients, relating them to ALE instantons and
Kac--Moody representation theory \cite{Nakajima}.

Hitchin's construction is the basic infinite-dimensional gauge-theoretic
analogue of the hyperkähler quotient theorem. The ambient space is the
configuration space of pairs consisting of a connection and a Higgs field, modeled on infinite-dimensional spaces of sections, and the unitary gauge group acts on it by
triholomorphic isometries. The self-duality equations are the corresponding
hyperkähler moment-map equations. Unlike the finite-dimensional case, however,
the quotient requires analytic input: one works with Sobolev completions and
uses the slice theorem, ellipticity, and regularity to obtain the smooth
hyperkähler moduli space \cite{Hitchin1}. This gauge-theoretic infinite-dimensional viewpoint
also underlies Konno's construction of parabolic Higgs bundle moduli spaces and
Fujiki's study of flat-bundle moduli in higher dimensional Kähler geometry
\cite{Konno,fujiki}.

For the purposes of the following discussion, we ignore the issue of whether the objects involved satisfy the assumptions of the hyperkähler quotient theorem and proceed at a purely formal level.

The NHYM moduli space considered here belongs to the infinite-dimensional setting. We therefore introduce related frameworks.
The space of all connections \(\mathcal A\) is an affine space modeled on the complex vector space
$
\Omega^1(\End E)
$. Choose a Hermitian metric \(h\) on
\(E\). The decomposition
$
\Omega^1(\End E)
=
\Omega^{0,1}(\End E)\oplus \Omega^{1,0}(\End E)
$
allows one to define a quaternionic structure on \(\Omega^1(M,\End E)\). Together
with the natural trace metric and a suitable Sobolev completion, this structure makes \(\mathcal A\) an
infinite-dimensional hyperkähler manifold.

Kaledin and Verbitsky suggest that the map
$
\mu_\mathbb C(\nabla)= \Lambda \Theta
$ on $\mathcal A$ 
should be regarded as a complex moment map. This leads formally to the expectation
that the quotient \(\mu_\mathbb C^{-1}(0)/ G\) should carry a hyperkähler structure by the
principle of hyperkähler reduction. However, this argument has to be interpreted
with care. A genuine hyperkähler quotient involves not only the complex moment
map, but also the real moment map, and the quotient is taken by a unitary
gauge group rather than by the full complex gauge group. Thus the formal quotient
\(\mu_\mathbb C^{-1}(0)/ G\) is not automatically a hyperkähler quotient in the
standard sense. For this reason, the hyperkähler quotient picture must be
reformulated more carefully.

Define the quaternionic structure on $\Omega^1(\End E)$ by 
\begin{equation}
I(a,b)=(ia,ib),\
J(a,b)=\bigl(i b^{*,h},-i a^{*,h}\bigr),\
K(a,b)=\bigl(-b^{*,h},a^{*,h}\bigr).
\end{equation}

It is useful to emphasize a basic difference between the NHYM and HYM settings.
In the NHYM case, the \((1,0)\) and \((0,1)\)-parts are
independent variables. This additional freedom makes it possible, at least formally,
to define a quaternionic triple \(I,J,K\) on the ambient space of connections. By
contrast, in the HYM setting one works with unitary connections with respect to a
fixed Hermitian metric. Hence the \((1,0)\)-part is determined by the \((0,1)\)-part
through the Hermitian adjoint. The same flat hyperkähler ambient model is not
available in this formulation.

The complex moment map is \(\mu_\mathbb C=\Lambda \nabla^2\) and 
the real moment map is \(\mu_\mathbb R=\Lambda \Xi_{\nabla,h}\) \cite[pp. 41--42]{NHYM}. We use \(U(E,h)\) to denote the unitary gauge group with respect to the
fixed Hermitian metric \(h\), because \(\mu_{\mathbb R}^{-1}(0)\) is \(U(E,h)\)-invariant.
Hence the formal hyperkähler quotient is \begin{equation}\mathcal M^{hk}=(\mu^{-1}_\mathbb C(0)\cap \mu^{-1}_\mathbb R(0))/U(E,h).\end{equation}

An important consequence of Theorem~\ref{ha} is that it provides, by choosing harmonic metrics, a set-theoretic injection
\begin{equation}
\iota:\mathcal M^s \hookrightarrow \mathcal M^{hk} .
\end{equation}

For a representative $\nabla \in [\nabla]_{G} \in \mathcal{M}^s$, we choose the harmonic metric $h_\nabla$ such that $\Lambda \Xi_{\nabla,h_{\nabla}}=0$, and there exists $q_\nabla \in \Gamma(\Herm^+_{h}(E))$ with $h_\nabla=(q_\nabla)^* h$, i.e., $h_\nabla(s,t)=h(q_\nabla s,q_\nabla t)$. Since $(\nabla,h_\nabla)$ is a harmonic pair, the transformed pair $(q_\nabla^{-1}\cdot \nabla ,(q_\nabla^{-1})^* h_\nabla)$ is also harmonic. Moreover, $(q_\nabla^{-1})^* h_\nabla=h$ and $\mu_\mathbb R(q_\nabla^{-1}\cdot \nabla)=\Lambda \Xi_{q_\nabla^{-1}\nabla,h}=0$. Thus $[q_\nabla^{-1}\cdot \nabla]_{U(E,h)} \in \mathcal M^{hk}$.

We then define $\iota([\nabla]_{G})=[q_\nabla^{-1} \cdot \nabla ]_{U(E,h)}$. It is well-defined. First, if one chooses another $\widehat{h}_\nabla=c h_\nabla$ with $\Lambda \Xi_{\nabla,\widehat{h}_\nabla}=0$, then $\widehat{h}_\nabla=\widehat{q}^*_\nabla h=(\sqrt{c}q_\nabla)^* h$. Hence $\widehat{q}^{-1}_\nabla \cdot \nabla=(\sqrt{c}q_\nabla)^{-1}\cdot \nabla=q_\nabla^{-1}\cdot \nabla$. Therefore $[q_\nabla^{-1} \cdot \nabla ]_{U(E,h)}=[\widehat{q}_\nabla^{-1} \cdot \nabla ]_{U(E,h)}$. Second, if $\nabla'=u \cdot \nabla \in [\nabla]_{G}$ for some $u \in G$, then $(\nabla',u^* h_\nabla)$ is harmonic. By the uniqueness, $h_{\nabla'}=c u^* h_\nabla=c u^* q^*_\nabla h =(\sqrt{c}q_\nabla u)^* h$. Thus $q_{\nabla'}$ and $\sqrt{c}q_\nabla u$ induce the same Hermitian metric from $h$. There exists $v \in U(E,h)$ such that $q_{\nabla'}=v \sqrt{c}q_\nabla u$. Therefore $q_{\nabla'}^{-1}\cdot \nabla'=(q^{-1}_{\nabla}v^{-1})\cdot \nabla=v^{-1}\cdot (q^{-1}_\nabla \cdot \nabla)$. It implies $[q_\nabla^{-1} \cdot \nabla ]_{U(E,h)}=[q_{\nabla'}^{-1} \cdot \nabla' ]_{U(E,h)}$. Here we use the convention that \(a\cdot(b\cdot\nabla)=(ba)\cdot\nabla\).

To see the injectivity, if $\iota([\nabla_1]_G)=\iota([\nabla_2]_G)$, then there exists \(v\in U(E,h)\) such that $ q_{\nabla_2}^{-1}\cdot\nabla_2 = v\cdot(q_{\nabla_1}^{-1}\cdot\nabla_1)= (q_{\nabla_1}^{-1}v)\cdot\nabla_1.$ Therefore, $\nabla_2=(q_{\nabla_1}^{-1}v q_{\nabla_2})\cdot\nabla_1$ where $q_{\nabla_1}^{-1}v q_{\nabla_2} \in G$. It follows that $[\nabla_1]_G=[\nabla_2]_G$.

The actual NHYM condition is stronger. Besides the moment-map
equation, it also requires the curvature-type conditions
$
\Theta^{2,0}=0,
\Theta^{0,2}=0
$. Thus this map is not expected to be surjective in general. 
\(\mathcal M^s\) should be regarded, at most, as a proper subspace of \(\mathcal M^{hk}\) via its image.

The question is therefore whether \(\mathcal M^s\) can inherit a hyperkähler structure from the ambient space. The following two examples illustrate both possibilities. The point is whether the subspace is invariant under the ambient hypercomplex structure. In the NHYM case, the additional curvature-type conditions are not preserved by the full hypercomplex structure, and hence \(\mathcal M^s\) cannot inherit the ambient hyperkähler structure in this direct way.

In Fujiki's work on flat
bundles over compact Kähler manifolds, the desired finite-dimensional moduli
space is obtained as a distinguished subspace of a larger hyperkähler quotient.
The hyperkähler structure restricts to this subspace precisely because its
tangent spaces are preserved by the three ambient complex structures
\cite{fujiki}. 

By contrast, in Dey's study of the dimensionally reduced
Seiberg--Witten equations with a Higgs field, a subset of the equations gives a
hyperkähler moduli space, but the moduli space defined by the full system does
not inherit this hyperkähler structure. Nevertheless, it still carries a natural
symplectic structure, together with an almost complex structure \cite{Dey}.
This comparison is relevant for the NHYM problem: even if the natural ambient
construction does not yield a hyperkähler metric on the NHYM moduli space, the
existence of a natural symplectic structure remains compatible with this general
pattern.

Besides the hyperkähler quotient method, another important source of
hyperkähler metrics is the twistor construction. Hitchin, Karlhede,
Lindström, and Roček explain how suitable twistor data can be used to
reconstruct hyperkähler manifolds \cite[Theorem 3.3]{hklr}. Feix proved that the cotangent bundle of
any real-analytic Kähler manifold carries a canonical hyperkähler metric in a
neighborhood of the zero section \cite{feix}. More recently, Mayrand
generalized this local construction: in particular, he showed that every
holomorphic symplectic groupoid over a compact holomorphic Poisson surface of
Kähler type admits a hyperkähler structure in a neighborhood of its identity
section \cite{Mayrand}. 
  This suggests a different possible approach to the NHYM problem: if the NHYM moduli space near the Hermitian locus could be identified with an appropriate twistor model, then the above construction might provide another route to a hyperkähler metric. We leave this possible approach for future work.

\bibliographystyle{unsrt}
\bibliography{sn-article}

@article {NHYM,
    AUTHOR = {Kaledin, D. and Verbitsky, M.},
     TITLE = {Non-{H}ermitian {Y}ang-{M}ills connections},
   JOURNAL = {Selecta Math. (N.S.)},
  FJOURNAL = {Selecta Mathematica. New Series},
    VOLUME = {4},
      YEAR = {1998},
    NUMBER = {2},
     PAGES = {279--320},
      ISSN = {1022-1824,1420-9020},
   MRCLASS = {53C07 (32G13 53C26 53C28 53C55)},
  MRNUMBER = {1669956},
MRREVIEWER = {Antony\ Maciocia},
       DOI = {10.1007/s000290050033},
       URL = {https://doi.org/10.1007/s000290050033},
}

@article {Itoh2,
    AUTHOR = {Itoh, Mitsuhiro},
     TITLE = {Geometry of anti-self-dual connections and {K}uranishi map},
   JOURNAL = {J{.} Math. Soc. Japan},
  FJOURNAL = {Journal of the Mathematical Society of Japan},
    VOLUME = {40},
      YEAR = {1988},
    NUMBER = {1},
     PAGES = {9--33},
      ISSN = {0025-5645,1881-1167},
   MRCLASS = {53C05 (14D20 32G13 53C25 58D17 58E15)},
  MRNUMBER = {917392},
MRREVIEWER = {P.\ E.\ Newstead},
       DOI = {10.2969/jmsj/04010009},
       URL = {https://doi.org/10.2969/jmsj/04010009},
}

@book {koba1,
    AUTHOR = {Kobayashi, Shoshichi},
     TITLE = {Differential geometry of complex vector bundles},
    SERIES = {Publications of the Mathematical Society of Japan},
    VOLUME = {15},
 PUBLISHER = {Princeton University Press},
 ADDRESS   = {Princeton, NJ},
      YEAR = {1987},
     PAGES = {xii+305},
      ISBN = {0-691-08467-X},
   MRCLASS = {53C55 (32-02 32L05 32L10 32L20)},
  MRNUMBER = {909698},
MRREVIEWER = {Daniel\ M.\ Burns, Jr.},
       DOI = {10.1515/9781400858682},
}

@article {Don,
    AUTHOR = {Donaldson, S. K.},
     TITLE = {Anti-self-dual {Y}ang-{M}ills connections over complex
              algebraic surfaces and stable vector bundles},
   JOURNAL = {Proc. London Math. Soc. (3)},
  FJOURNAL = {Proceedings of the London Mathematical Society. Third Series},
    VOLUME = {50},
      YEAR = {1985},
    NUMBER = {1},
     PAGES = {1--26},
      ISSN = {0024-6115,1460-244X},
   MRCLASS = {58E15 (14F99 53C05 57R99)},
  MRNUMBER = {765366},
MRREVIEWER = {S.\ Ramanan},
       DOI = {10.1112/plms/s3-50.1.1},
       URL = {https://doi.org/10.1112/plms/s3-50.1.1},
}

@article {Sim1,
    AUTHOR = {Simpson, Carlos T.},
     TITLE = {Constructing variations of {H}odge structure using
              {Y}ang-{M}ills theory and applications to uniformization},
   JOURNAL = {J. Amer. Math. Soc.},
  FJOURNAL = {Journal of the American Mathematical Society},
    VOLUME = {1},
      YEAR = {1988},
    NUMBER = {4},
     PAGES = {867--918},
      ISSN = {0894-0347,1088-6834},
   MRCLASS = {58E15 (32L15 53C25 53C55)},
  MRNUMBER = {944577},
       DOI = {10.2307/1990994},
       URL = {https://doi.org/10.2307/1990994},
}

@article {zhang,
    AUTHOR = {Pan, Changpeng and Shen, Zhenghan and Zhang, Xi},
     TITLE = {On the existence of harmonic metrics on non-{H}ermitian
              {Y}ang-{M}ills bundles},
   JOURNAL = {Proc. Lond. Math. Soc. (3)},
  FJOURNAL = {Proceedings of the London Mathematical Society. Third Series},
    VOLUME = {128},
      YEAR = {2024},
    NUMBER = {1},
     PAGES = {Paper No. e12580, 24},
      ISSN = {0024-6115,1460-244X},
   MRCLASS = {53C07},
  MRNUMBER = {4687029},
MRREVIEWER = {Yanci\ Nie},
       DOI = {10.1112/plms.12580},
       URL = {https://doi.org/10.1112/plms.12580},
}

@book{linear,
  title={Foundations of modern analysis},
  author={Friedman, Avner},
  year={1982},
  publisher = {Dover Publications},
  address   = {New York},
isbn      = {0-486-64062-0}
}

@article {hklr,
    AUTHOR = {Hitchin, N. J. and Karlhede, A. and Lindstr\"om, U. and Roček, M.},
     TITLE = {Hyper{k}\"ahler metrics and supersymmetry},
   JOURNAL = {Comm. Math. Phys.},
  FJOURNAL = {Communications in Mathematical Physics},
    VOLUME = {108},
      YEAR = {1987},
    NUMBER = {4},
     PAGES = {535--589},
      ISSN = {0010-3616,1432-0916},
   MRCLASS = {53C25 (32C10 53C55 53C80 58E99 81E99)},
  MRNUMBER = {877637},
MRREVIEWER = {Pankaj\ Topiwala},
DOI = {10.1007/BF01214418},
       URL = {http://projecteuclid.org/euclid.cmp/1104116624},
}

@article {Hitchin1,
    AUTHOR = {Hitchin, N. J.},
     TITLE = {The self-duality equations on a {R}iemann surface},
   JOURNAL = {Proc. Lond. Math. Soc. (3)},
  FJOURNAL = {Proceedings of the London Mathematical Society. Third Series},
    VOLUME = {55},
      YEAR = {1987},
    NUMBER = {1},
     PAGES = {59--126},
      ISSN = {0024-6115,1460-244X},
   MRCLASS = {32G13 (14F05 14H15 32L10 53C05 58E99 81E13)},
  MRNUMBER = {887284},
MRREVIEWER = {Mitsuhiro\ Itoh},
       DOI = {10.1112/plms/s3-55.1.59},
       URL = {https://doi.org/10.1112/plms/s3-55.1.59},
}

@article {Sim11,
    AUTHOR = {Simpson, Carlos T.},
     TITLE = {Moduli of representations of the fundamental group of a smooth
              projective variety. {I}},
   JOURNAL = {Inst. Hautes \'Etudes Sci. Publ. Math.},
  FJOURNAL = {Institut des Hautes \'Etudes Scientifiques. Publications
              Math\'ematiques},
    NUMBER = {79},
      YEAR = {1994},
     PAGES = {47--129},
      ISSN = {0073-8301,1618-1913},
   MRCLASS = {14D20 (14D22 14D25 14F05)},
  MRNUMBER = {1307297},
MRREVIEWER = {Nitin\ Nitsure},
DOI = {10.1007/BF02698887},
       URL = {http://www.numdam.org/item?id=PMIHES_1994__79__47_0},
}

@article {Sim22,
    AUTHOR = {Simpson, Carlos T.},
     TITLE = {Moduli of representations of the fundamental group of a smooth
              projective variety. {II}},
   JOURNAL = {Inst. Hautes \'Etudes Sci. Publ. Math.},
  FJOURNAL = {Institut des Hautes \'Etudes Scientifiques. Publications
              Math\'ematiques},
    NUMBER = {80},
      YEAR = {1994},
     PAGES = {5--79},
      ISSN = {0073-8301,1618-1913},
   MRCLASS = {14D20 (14D22 14F05 14F10)},
  MRNUMBER = {1320603},
MRREVIEWER = {Nitin\ Nitsure},
DOI = {10.1007/BF02698895},
       URL = {http://www.numdam.org/item?id=PMIHES_1994__80__5_0},
}

@article {UY,
    AUTHOR = {Uhlenbeck, K. and Yau, S.-T.},
     TITLE = {On the existence of {H}ermitian-{Y}ang-{M}ills connections in
              stable vector bundles},
      NOTE = {Frontiers of the mathematical sciences: 1985 (New York, 1985)},
   JOURNAL = {Comm. Pure Appl. Math.},
  FJOURNAL = {Communications on Pure and Applied Mathematics},
    VOLUME = {39},
      YEAR = {1986},
     PAGES = {S257--S293},
      ISSN = {0010-3640,1097-0312},
   MRCLASS = {58G05 (32L15 53C05 58E15)},
  MRNUMBER = {861491},
MRREVIEWER = {Daniel\ S.\ Freed},
       DOI = {10.1002/cpa.3160390714},
       URL = {https://doi.org/10.1002/cpa.3160390714},
}

@article {zhang2,
    AUTHOR = {Wu, Di and Zhang, Xi},
     TITLE = {Poisson metrics and {H}iggs bundles over noncompact {K}\"ahler
              manifolds},
   JOURNAL = {Calc. Var. Partial Differential Equations},
  FJOURNAL = {Calculus of Variations and Partial Differential Equations},
    VOLUME = {62},
      YEAR = {2023},
    NUMBER = {1},
     PAGES = {Paper No. 20, 29},
      ISSN = {0944-2669,1432-0835},
   MRCLASS = {53C07 (14J60 53C21)},
  MRNUMBER = {4505163},
MRREVIEWER = {Ivan\ Tulli},
       DOI = {10.1007/s00526-022-02343-z},
       URL = {https://doi.org/10.1007/s00526-022-02343-z},
}

@article {Cor,
    AUTHOR = {Corlette, Kevin},
     TITLE = {Flat {$G$}-bundles with canonical metrics},
   JOURNAL = {J. Differential Geom.},
  FJOURNAL = {Journal of Differential Geometry},
    VOLUME = {28},
      YEAR = {1988},
    NUMBER = {3},
     PAGES = {361--382},
      ISSN = {0022-040X,1945-743X},
   MRCLASS = {58E20 (32L99 53C10)},
  MRNUMBER = {965220},
MRREVIEWER = {John\ C.\ Wood},
DOI = {10.4310/jdg/1214442469},
       URL = {http://projecteuclid.org/euclid.jdg/1214442469},
}

@article {Don1,
    AUTHOR = {Donaldson, S. K.},
     TITLE = {Twisted harmonic maps and the self-duality equations},
   JOURNAL = {Proc. Lond. Math. Soc. (3)},
  FJOURNAL = {Proceedings of the London Mathematical Society. Third Series},
    VOLUME = {55},
      YEAR = {1987},
    NUMBER = {1},
     PAGES = {127--131},
      ISSN = {0024-6115,1460-244X},
   MRCLASS = {58E20 (32L15 53C05)},
  MRNUMBER = {887285},
MRREVIEWER = {Mitsuhiro\ Itoh},
       DOI = {10.1112/plms/s3-55.1.127},
       URL = {https://doi.org/10.1112/plms/s3-55.1.127},
}

@book {Ein,
    AUTHOR = {Besse, Arthur L.},
     TITLE = {Einstein manifolds},
    SERIES = {Ergebnisse der Mathematik und ihrer Grenzgebiete (3) [Results
              in Mathematics and Related Areas (3)]},
    VOLUME = {10},
 PUBLISHER = {Springer-Verlag},
ADDRESS   = {Berlin},
      YEAR = {1987},
     PAGES = {xii+510},
      ISBN = {3-540-15279-2},
   MRCLASS = {53C25 (53-02 53C21 53C30 53C55 58D17 58E11)},
  MRNUMBER = {867684},
MRREVIEWER = {S.\ M.\ Salamon},
       DOI = {10.1007/978-3-540-74311-8},
}

@incollection {fujiki,
    AUTHOR = {Fujiki, Akira},
     TITLE = {Hyper-{K}\"ahler structure on the moduli space of flat
              bundles},
 BOOKTITLE = {Prospects in complex geometry ({K}atata and {K}yoto, 1989)},
    SERIES = {Lecture Notes in Math.},
    VOLUME = {1468},
     PAGES = {1--83},
 PUBLISHER = {Springer},
ADDRESS   = {Berlin},
      YEAR = {1991},
      ISBN = {3-540-54053-9},
   MRCLASS = {53C07 (53C25 58D27 58E20)},
  MRNUMBER = {1123536},
MRREVIEWER = {N.\ J.\ Hitchin},
       DOI = {10.1007/BFb0086187},
}

@article {Kron1,
    AUTHOR = {Kronheimer, P. B.},
     TITLE = {The construction of {ALE} spaces as hyper-{K}\"ahler
              quotients},
   JOURNAL = {J. Differential Geom.},
  FJOURNAL = {Journal of Differential Geometry},
    VOLUME = {29},
      YEAR = {1989},
    NUMBER = {3},
     PAGES = {665--683},
      ISSN = {0022-040X,1945-743X},
   MRCLASS = {53C25 (14B07 32L10 32M10 53C55 83C20)},
  MRNUMBER = {992334},
MRREVIEWER = {Krzysztof\ Galicki},
DOI = {10.4310/jdg/1214443066},
       URL = {http://projecteuclid.org/euclid.jdg/1214443066},
}

@article {Nakajima,
    AUTHOR = {Nakajima, Hiraku},
     TITLE = {Instantons on {ALE} spaces, quiver varieties, and
              {K}ac–{M}oody algebras},
   JOURNAL = {Duke Math. J.},
  FJOURNAL = {Duke Mathematical Journal},
    VOLUME = {76},
      YEAR = {1994},
    NUMBER = {2},
     PAGES = {365--416},
      ISSN = {0012-7094,1547-7398},
   MRCLASS = {53C25 (17B67 58D27 58E15)},
  MRNUMBER = {1302318},
MRREVIEWER = {Andrew\ Dancer},
       DOI = {10.1215/S0012-7094-94-07613-8},
       URL = {https://doi.org/10.1215/S0012-7094-94-07613-8},
}

@article {Konno,
    AUTHOR = {Konno, Hiroshi},
     TITLE = {Construction of the moduli space of stable parabolic {H}iggs
              bundles on a {R}iemann surface},
   JOURNAL = {J. Math. Soc. Japan},
  FJOURNAL = {Journal of the Mathematical Society of Japan},
    VOLUME = {45},
      YEAR = {1993},
    NUMBER = {2},
     PAGES = {253--276},
      ISSN = {0025-5645,1881-1167},
   MRCLASS = {32G13 (14D20 32L07 53C07 58D27)},
  MRNUMBER = {1206652},
MRREVIEWER = {Steven\ B.\ Bradlow},
       DOI = {10.2969/jmsj/04520253},
       URL = {https://doi.org/10.2969/jmsj/04520253},
}

@article {Dey,
    AUTHOR = {Dey, Rukmini},
     TITLE = {Symplectic and hyperk\"ahler structures in a dimensional
              reduction of the {S}eiberg-{W}itten equations with a {H}iggs
              field},
   JOURNAL = {Rep. Math. Phys.},
  FJOURNAL = {Reports on Mathematical Physics},
    VOLUME = {50},
      YEAR = {2002},
    NUMBER = {3},
     PAGES = {277--290},
      ISSN = {0034-4877},
   MRCLASS = {53C26 (53D30 57R57 58D27)},
  MRNUMBER = {1952133},
MRREVIEWER = {Liviu\ I.\ Nicolaescu},
       DOI = {10.1016/S0034-4877(02)80058-1},
       URL = {https://doi.org/10.1016/S0034-4877(02)80058-1},
}

@article {feix,
    AUTHOR = {Feix, Birte},
     TITLE = {Hyperk\"ahler metrics on cotangent bundles},
   JOURNAL = {J. Reine Angew. Math.},
  FJOURNAL = {Journal f\"ur die Reine und Angewandte Mathematik. [Crelle's
              Journal]},
    VOLUME = {532},
      YEAR = {2001},
     PAGES = {33--46},
      ISSN = {0075-4102,1435-5345},
   MRCLASS = {53C26 (53C28)},
  MRNUMBER = {1817502},
MRREVIEWER = {Roger\ Bielawski},
       DOI = {10.1515/crll.2001.017},
       URL = {https://doi.org/10.1515/crll.2001.017},
}

@article {Mayrand,
    AUTHOR = {Mayrand, Maxence},
     TITLE = {Hyperk\"ahler metrics near {L}agrangian submanifolds and
              symplectic groupoids},
   JOURNAL = {J. Reine Angew. Math.},
  FJOURNAL = {Journal f\"ur die Reine und Angewandte Mathematik. [Crelle's
              Journal]},
    VOLUME = {782},
      YEAR = {2022},
     PAGES = {197--218},
      ISSN = {0075-4102,1435-5345},
   MRCLASS = {53C26 (32Q15 53D12)},
  MRNUMBER = {4360006},
MRREVIEWER = {Shangrong\ Deng},
       DOI = {10.1515/crelle-2021-0059},
       URL = {https://doi.org/10.1515/crelle-2021-0059},
}

@article {Tay,
    AUTHOR = {Del Moral, P. and Niclas, A.},
     TITLE = {A {T}aylor expansion of the square root matrix function},
   JOURNAL = {J. Math. Anal. Appl.},
  FJOURNAL = {Journal of Mathematical Analysis and Applications},
    VOLUME = {465},
      YEAR = {2018},
    NUMBER = {1},
     PAGES = {259--266},
      ISSN = {0022-247X,1096-0813},
   MRCLASS = {65F60 (41A58)},
  MRNUMBER = {3806701},
       DOI = {10.1016/j.jmaa.2018.05.005},
       URL = {https://doi.org/10.1016/j.jmaa.2018.05.005},
}

@article {Te,
    AUTHOR = {Arbarello, Enrico and Cornalba, Maurizio},
     TITLE = {Teichm\"uller space via {K}uranishi families},
   JOURNAL = {Ann. Sc. Norm. Super. Pisa Cl. Sci. (5)},
  FJOURNAL = {Annali della Scuola Normale Superiore di Pisa. Classe di
              Scienze. Serie V},
    VOLUME = {8},
      YEAR = {2009},
    NUMBER = {1},
     PAGES = {89--116},
      ISSN = {0391-173X,2036-2145},
   MRCLASS = {32G15},
  MRNUMBER = {2512202},
MRREVIEWER = {Kimio\ Miyajima},
DOI = {10.2422/2036-2145.2009.1.05},
}

\end{document}